\documentclass[leqno]{amsart}

\usepackage{amsmath,amssymb,amsthm}
\usepackage{mathrsfs}
\usepackage{braket}
\usepackage{graphicx}
\usepackage{amscd}
\usepackage[all]{xy}
\usepackage{eucal}

\theoremstyle{definition}
\newtheorem{dfn}{Definition}[section]
\newtheorem{lem}[dfn]{Lemma}
\newtheorem{cor}[dfn]{Corollary}
\newtheorem{prp}[dfn]{Proposition}
\newtheorem{theom}[dfn]{Theorem}
\newtheorem{rem}[dfn]{Remark}
\newtheorem{fct}[dfn]{Fact}
\newtheorem{ex}[dfn]{Example}
\newcommand{\mf}{\mathfrak}
\DeclareMathOperator{\End}{\sf End}

\DeclareMathOperator{\Hom}{\sf Hom}
\DeclareMathOperator{\tr}{\sf tr}
\DeclareMathOperator{\ch}{\sf ch}
\DeclareMathOperator{\id}{\sf id}
\DeclareMathOperator{\Ind}{Ind}

\DeclareMathOperator{\Obj}{\sf Obj}
\makeatletter

\@addtoreset{equation}{section}
\makeatother

\title[Equivalences between weight modules]
{Equivalences between weight modules
via $\mathcal{N}=2$ coset constructions}
\author{Ryo SATO}
\address{Graduate School of Mathematical Sciences,\\
The University of Tokyo \\
3-8-1 Komaba Meguro-ku Tokyo,\\
153-8914, Japan}
\email{rsato@ms.u-tokyo.ac.jp}

\keywords{Representation theory, vertex superalgebra, superconformal algebra}
\subjclass{17B10, 17B65, 17B69}

\begin{document}

\maketitle

\begin{abstract}
In this paper we introduce a variant of
weight modules for certain conformal vertex superalgebras
as an appropriate framework
of the $\mathcal{N}=2$ supersymmetric coset construction.
We call them weight-wise admissible modules.
Motivated by the work of Feigin-Semikhatov-Tipunin,
we give (block-wise) categorical equivalences between 
the categories of weight-wise admissible modules over $\widehat{\mf{sl}}_{2}$ and
the $\mathcal{N}=2$ superconformal algebra,
induced by the coset construction.

As an application, we obtain some character formulas
of modules over the $\mathcal{N}=2$ superconformal algebra.
\end{abstract}

\section{Introduction}

\subsection{Background}
The infinitesimal symmetry of
$\mathcal{N}=2$ supersymmetric conformal field theories
is desicribed by the $\mathcal{N}=2$ superconformal algebra (SCA)
introduced in \cite{adem76}, 
which is a supersymmetric generalization of the Virasoro algebra.
There are two inequivalent classes of the $\mathcal{N}=2$ SCA
called the untwisted sector and the twisted sector.
In this paper we only deal with the untwisted sector.
See \cite{iohara2010unitarizable} and references therein for the details.

In 1980s many physicists extensively studied unitarizable highest weight
modules over the $\mathcal{N}=2$ SCA
in order to construct models of rational supersymmetric conformal field theories.
Among those modules, the classification of irreducible ones 
is conjectured by \cite{boucher1986determinant},
and its mathematically rigorous proof is given by \cite{iohara2010unitarizable} 
for the untwisted sector.
In its proof the $\mathcal{N}=2$ supersymmetric coset construction 
\cite{DPYZ86} (and \cite{KS89} for general affine Lie algebras)
 plays an important role to construct some specific unitarizable modules, 
the so-called minimal unitary series of the $\mathcal{N}=2$ SCA.
In \cite{DPYZ86} they constructed
all the minimal unitary series as direct summands of the tensor products of 
unitarizable highest weight $\widehat{\mf{sl}}_{2}$-modules
and a free fermionic Fock module.
This construction is analogous to the Goddard-Kent-Olive coset constructions
for the case of $\mathcal{N}=0$ or $1$ in \cite{GKO86}
and can be reformulated
in terms of vertex superalgebras.

For a vertex superalgebra $V$ and its vertex subsuperalgebra $W$,
we denote by ${\sf Com}(W,V)$ the commutant 
vertex superalgebra of $W$ in $V$ (see Appendix \ref{Cvertex superalgebra} for the definition).
Assume that $V$ is the tensor product of
the affine vertex operator algebra (VOA) 
associated with $\widehat{\mf{sl}}_{2}$
and the free fermionic vertex operator superalgebra (VOSA), 
and $W$ is a certain vertex subsuperalgebra which is isomorphic to 
the Heisenberg VOA of rank $1$ (see Remark \ref{HeisVOA} for the definition). 
Then it is known that the commutant vertex superalgebra ${\sf Com}(W,V)$ 
is isomorphic to the $\mathcal{N}=2$ VOSA (see \cite[\S8.2]{CREUTZIG2019396} and Corollary \ref{CST}).
It is also known that there exists a vertex superalgebra homomorphism
${\sf Com}(W,V)\otimes W\to V;\,A\otimes B\mapsto A_{(-1)}B$
(see e.g.\,\cite[Proposition 4.4.1]{matsuo17036axioms}).
In particular, the above $\mathcal{N}=2$ coset construction
is regarded as a theory of branching rules with respect to the above homomorphism
 when both the affine VOA and the corresponding $\mathcal{N}=2$ VOSA
are regular in the sense of \cite{dong1997regularity}.
See \cite{Ad01} and \cite[Example 5]{creutzig2016schur} for detail.

On the contrary, if these VOSAs are not regular, the complete reducibility 
of modules fails in general, and
the pullback action induced by
the above homomorphism is hard to analyze.
In \cite{FST98} B. Feigin, A. Semikhatov and I. Tipunin discovered another
coset construction, which is given by interchanging the roles
of the affine VOA and the $\mathcal{N}=2$ VOSA
(see also \cite{Ad99}).
Namely, they assume that $V$ is the tensor product of the $\mathcal{N}=2$ VOSA
and a certain negative-definite lattice vertex superalgebra, and
$W$ is again a certain vertex subsuperalgebra which is isomorphic to the Heisenberg VOA of rank $1$.
Then the commutant vertex superalgebra ${\sf Com}(W,V)$ is isomorphic to 
the affine VOA associated with $\mf{sl}_{2}$ (see Corollary \ref{CST}).
By using these bidirectional constructions,
they investigated an `equivalence' between 
more general weight modules over $\widehat{\mf{sl}}_{2}$
and over the $\mathcal{N}=2$ SCA, which works even for 
the irregular cases.
See \cite[Section IV]{FST98} for the details.
However, mathematically rigorous descriptions of the
`equivalence' in \cite{FST98} and explicit proofs are not given in the literature
to the best of our knowledge.

\subsection{Main result}
In this paper we introduce a natural framework 
including all the weight modules considered in \cite{FST98} (cf.\,\cite{creutzig2016schur}).
First, we assume that a conformal vertex superalgebra $V$ contains
a vertex subalgebra which is isomorphic to
a Heisenberg vertex algebra $\mathcal{F}$ (see \S \ref{general}),
and fix such a subalgebra.
Roughly speaking, we call a weak $V$-module $M$ 
a \emph{weight-wise admissible module}
if $M$ decomposes into a direct sum of Heisenberg Fock modules
as a weak $\mathcal{F}$-module (see Section 3 for the precise definition).
Note that $M$ is admissible as a weak $\mathcal{F}$-module 
in the sense of \cite{dong1997regularity}.
We denote the category of such modules by $\mathcal{C}_{V}$.

When $V$ is an affine VOA associated with $\widehat{\mf{sl}}_{2}$ 
or an $\mathcal{N}=2$ VOSA, 
it contains the Heisenberg vertex algebra of rank $1$.
We can verify that the category $\mathcal{C}_{V}$
decomposes into a direct sum of certain full subcategories
which we call \emph{blocks} (Definition \ref{block}).
Now we state the main result: 
we construct the functors which establish
equivalences as $\mathbb{C}$-linear abelian categories
between such blocks of $\widehat{\mf{sl}}_{2}$ and the $\mathcal{N}=2$ SCA
(Theorem \ref{main} and \ref{mainsimple}).
In addition, we prove that these functors are ``spectral flow equivariant'' (Corollary \ref{sfeq}).

As an application, we obtain
new character formulas of
weight-wise admissble modules over the $\mathcal{N}=2$ SCA (Theorem \ref{cForm}).
More precisely, these formulas are written
in terms of the characters of the corresponding 
$\widehat{\mf{sl}}_{2}$-modules by the previous equivalence.
In particular, we succeed to reprove
the character formulas of the `admissible' $\mathcal{N}=2$ modules
of central charge $c=3\left(1-\frac{2p'}{p}\right)
\ (p\geq2,\,p'\geq1)$
in the sense of \cite{FSST99} (see Theorem \ref{FSSTCF}). 

Note that, if the VOSAs are regular, weight-wise admissible modules 
over the affine simple VOA and the $\mathcal{N}=2$ simple VOSA
coincide with (ordinary) modules 
over VOSAs in the sense of \cite{kac1994vertex}.
In this sense, the notion of weight-wise admissible modules is
a natural generalization of that of ordinary modules over VOSAs.

\subsection{Further problems}
Let us explain some issues which are related with our results.
Among non-unitary modules over affine Lie algebras, 
the Kac-Wakimoto admissible modules at 
admissible levels are distinguished by
their modular invariant properties \cite{KW88}.
At the admissible levels,
weight-wise admissible modules naturally appear in the study 
of ``fusion products'' from a viewpoint of \cite{G01}
(see also \cite{feigin1994fusion}, \cite{dong1997vertex}).
The Wess-Zumino-Witten models at these levels are believed to relate with
the so-called logarithmic conformal field theories,
which include non-simple indecomposable modules and 
allow logarithmic singularities 
in the corresponding conformal blocks.
For instance, the `logarithmic' indecomposable module
over $\widehat{\mf{sl}}_{2}$
constructed in \cite[Proposition 8.1]{adamovic2009lattice}
gives a specific example of weight-wise admissible modules. 

Unfortunately, the simple affine VOAs at the admissible levels
are not $C_{2}$-cofinite in general
(see \cite[Proposition 3.4.1]{AM95}
for the $\widehat{\mf{sl}}_{2}$-case).
In particular, the characters of some irreducible non-highest weight modules
over the admissble affine VOA are not convergent 
(see \cite[Corollary 5]{creutzig2013modular} for example).
Contrary to this, by our character formula in Theorem \ref{cForm},
we can see that
the corresponding characters in the $\mathcal{N}=2$ side are
absolutely convergent in some region (see Remark \ref{divergent} and Proposition \ref{cpair}).
We expect that the study of the modular property of these functions may be
the first step to reveal the relationship between 
the modular invariance and the fusion rules at the admissible levels,
for example, a modification of Verlinde's formula initiated by \cite{awata1992fusion}
(see also \cite{G01}, \cite{creutzig2013modular}).

\subsection{Structure of the paper}
We organize this paper as follows.
In Section 2, after recalling some relevant facts about 
vertex superalgebras,
we give a brief review of the $\mathcal{N}=2$ supersymmetric
coset constructions.
Section 3 is devoted to the definition of weight-wise admissible module
over conformal vertex superalgebras with certain additional conditions.
We present our main result, categorical equivalences between blocks, in Section 4
and prove it in Section 5.
In Section 6, we discuss a relationship
between the functors constructed in Section 4
and the spectral flow automorphisms
of $\widehat{\mf{sl}}_{2}$ and the $\mathcal{N}=2$ SCA
(see Appendix \ref{SFA} for the definition).
In Section 7, we present some applications of the main result.

\medskip
{\bf Acknowledgments: }
The author would like to
express his gratitude to Tomoyuki Arakawa,
Kenji Iohara and Yoshihisa Saito 
for helpful discussions and comments.
He also thanks Hironori Oya for many helpful discussions
and Sebastiano Carpi for letting him know the references.
This work is supported by the Program for Leading Graduate
Schools, MEXT, Japan.

\section{Preliminaries}\label{Prelim}

In this section we recall some basic facts about vertex superalgebras
(see the book \cite{Kac97V} for details) and 
the vertex superalgebraic formulation of the $\mathcal{N}=2$
supersymmetric coset constructions.

\subsection{VOA associated with $\widehat{\mf{sl}}_{2}$}

Let $\mf{sl}_{2}$ be 
the simple Lie algebra of type $A_{1}$ and
$$E=
\begin{pmatrix}
0 & 1 \\
0 & 0
\end{pmatrix}
,\ 
H=
\begin{pmatrix}
1 & 0 \\
0 & -1
\end{pmatrix}
,\ 
F=
\begin{pmatrix}
0 & 0 \\
1 & 0
\end{pmatrix}
$$
the standard basis of $\mf{sl}_{2}$.
The normalized invariant form $B_{0}$ on $\mf{sl}_{2}$ is given by
$B_{0}(X,Y)=\tr(XY)$ for $X, Y\in\mf{sl}_{2}$.
The affine Lie algebra $\widehat{\mf{sl}}_{2}$ is nothing but the affinization
$(\widehat{\mf{sl}}_{2})_{B_{0}}$ defined in \S \ref{Vaff}.

Throughout this paper, 
we always assume $k\neq-2$ and
use the following notations:
\begin{equation*}
\bigl(V_{k}(\mf{sl}_{2}),L_{k}(\mf{sl}_{2}),\mathbf{1}^{k},\omega^{\mf{sl}_{2}}\bigr)
:=\bigl(V_{kB_{0}}(\mf{sl}_{2}),L_{kB_{0}}(\mf{sl}_{2}),\mathbf{1}_{kB_{0}},\omega_{\mf{sl}_{2},k}\bigr).
\end{equation*}
See \S \ref{Vaff} for the definition.

To simplify notation, we write $\mathbf{H},\mathbf{E}$ and 
$\mathbf{F}$ for $H_{-1}{\bf 1}^{k},E_{-1}{\bf 1}^{k}$ 
and $F_{-1}{\bf 1}^{k}$, respectively.
We also write $L^{\mf{sl}_{2}}_{n}$
for the linear operator $\omega^{\mf{sl}_{2}}_{(n+1)}$.

For $j\in\mathbb{C}$, let $M_{j,k}$
be the Verma module of highest weight $(k-2j)\Lambda_{0}+2j\Lambda_{1}$ and
$\ket{j,k}$ the highest weight vector of $M_{j,k}$,
where $\Lambda_{i}$ is the $i$-th fundamental weight of $\widehat{\mf{sl}}_{2}$.
It is clear that $V_{k}(\mf{sl}_{2})$ is a quotient $\widehat{\mf{sl}}_{2}$-module of $M_{0,k}$
and $L_{k}(\mf{sl}_{2})$ coincides with the simple quotient of $M_{0,k}$.

\subsection{VOSA associated with the $\mathcal{N}=2$ SCA}\label{Vvir}

The Neveu-Schwarz sector of the $\mathcal{N}=2$ 
superconformal algebra is the Lie superalgebra
\begin{equation*}
\mf{ns}_{2}=
\bigoplus_{n\in\mathbb{Z}}\mathbb{C}
L_{n}\oplus
\bigoplus_{n\in\mathbb{Z}}\mathbb{C}
J_{n}\oplus
\bigoplus_{r\in\mathbb{Z}+\frac{1}{2}}\mathbb{C}
G^{+}_{r}\oplus
\bigoplus_{r\in\mathbb{Z}+\frac{1}{2}}\mathbb{C}
G^{-}_{r}\oplus
\mathbb{C}C
\end{equation*}
whose $\mathbb{Z}_{2}$-grading is given by 
\begin{equation*}
(\mf{ns}_{2})^{\bar{0}}=
\bigoplus_{n\in\mathbb{Z}}\mathbb{C}
L_{n}\oplus
\bigoplus_{n\in\mathbb{Z}}\mathbb{C}
J_{n}\oplus\mathbb{C}C,\ \ 
(\mf{ns}_{2})^{\bar{1}}=
\bigoplus_{r\in\mathbb{Z}+\frac{1}{2}}\mathbb{C}
G^{+}_{r}\oplus
\bigoplus_{r\in\mathbb{Z}+\frac{1}{2}}\mathbb{C}
G^{-}_{r}
\end{equation*}
with the following 
(anti-)commutation relations:
\begin{flalign*}
&
[L_{n},L_{m}]=(n-m)L_{n+m}+\frac{1}{12}(n^{3}-n)C\delta_{n+m,0},\\ 
&
[L_{n},J_{m}]=-mJ_{n+m}, \ 
[L_{n},G^{\pm}_{r}]=\left(\frac{n}{2}-r\right)G^{\pm}_{n+r}, \\
&
[J_{n},J_{m}]=\frac{n}{3}C\delta_{n+m,0},\ \ 
[J_{n},G^{\pm}_{r}]=\pm G^{\pm}_{n+r}, \\
&
[G^{+}_{r},G^{-}_{s}]=
2L_{r+s}+(r-s)J_{r+s}+
\frac{1}{3}\left(r^{2}-\frac{1}{4}\right)C\delta_{r+s,0}, \\
&
[G^{+}_{r},G^{+}_{s}]=
[\,G^{-}_{r},G^{-}_{s}]=0,\ \ 
[\mf{ns}_{2},C]=\{0\},
\end{flalign*}
for $n,m\in\mathbb{Z}$ and $r,s\in\mathbb{Z}+\frac{1}{2}$.

Let $\mf{ns}_{2}=
(\mf{ns}_{2})_{+}\oplus
(\mf{ns}_{2})_{0}\oplus
(\mf{ns}_{2})_{-}$
be the triangular decomposition of $\mf{ns}_{2}$,
where 
\begin{equation*}
\begin{array}{lcl}
(\mf{ns}_{2})_{+} & := &
\displaystyle
\bigoplus_{n>0}\mathbb{C}L_{n}\oplus
\bigoplus_{n>0}\mathbb{C}J_{n}\oplus
\bigoplus_{r>0}\mathbb{C}G^{+}_{r}\oplus
\bigoplus_{r>0}\mathbb{C}G^{-}_{r},\\
(\mf{ns}_{2})_{-} & := &
\displaystyle
\bigoplus_{n<0}\mathbb{C}L_{n}\oplus
\bigoplus_{n<0}\mathbb{C}J_{n}\oplus
\bigoplus_{r<0}\mathbb{C}G^{+}_{r}\oplus
\bigoplus_{r<0}\mathbb{C}G^{-}_{r}
,\\
(\mf{ns}_{2})_{0} & := &
\displaystyle
\mathbb{C}L_{0}\oplus
\mathbb{C}J_{0}
\oplus\mathbb{C}C,
\end{array}
\end{equation*}
and set $(\mf{ns}_{2})_{\geq0}:=
(\mf{ns}_{2})_{+}\oplus(\mf{ns}_{2})_{0}$.
For $h,j,c\in\mathbb{C}$, let $\mathbb{C}_{h,j,c}$
be the 1-dimensional $(\mf{ns}_{2})_{\geq0}$-module
defined by $(\mf{ns}_{2})_{+}.1:=\{0\},\ 
L_{0}.1:=h,\ J_{0}.1:=j$ and
$C.1:=c$.
We call the induced module 
$\mathcal{M}_{h,j,c}:=
\Ind_{(\mf{ns}_{2})_{\geq0}}^{\mf{ns}_{2}}
\mathbb{C}_{h,j,c}$
 the \emph{Verma module} 
of $\mf{ns}_{2}$.
Let us denote 
by $\ket{h,j,c}^{\mf{ns}_{2}}$ 
the canonical generator $1\otimes1\in\mathcal{M}_{h,j,c}$
and by $\mathcal{L}_{h,j,c}$ the simple quotient of $\mathcal{M}_{h,j,c}$.

When $h=\pm\frac{j}{2}$, 
the quotient modules 
$$\textstyle
\mathcal{M}^{\pm}_{j,c}:=\mathcal{M}_{\pm\frac{j}{2},j,c}
\big/U(\mf{ns}_{2})G^{\pm}_{-\frac{1}{2}}
\ket{\pm\frac{j}{2},j,c}^{\mf{ns}_{2}}$$
are called
the \emph{chiral Verma module}
and the \emph{anti-chiral Verma module}.
If $h=j=0$, according to \cite[Proposition 1.1]{Ad99}, the quotient module 
$$V_{c}(\mf{ns}_{2}):=
\mathcal{M}_{0,0,c}\big/\bigl(U(\mf{ns}_{2})
G^{+}_{-\frac{1}{2}}\ket{0,0,c}^{\mf{ns}_{2}}
+U(\mf{ns}_{2})
G^{-}_{-\frac{1}{2}}\ket{0,0,c}^{\mf{ns}_{2}}\bigr)$$
has a VOSA structure of central charge $c$,
where the vacuum ${\bf 1}^{c}$ is given by
the image of $\ket{0,0,c}^{\mf{ns}_{2}}$ 
in $V_{c}(\mf{ns}_{2})$ and
the conformal vector $\omega^{\mf{ns}_{2}}$ is given by $L_{-2}{\bf 1}^{c}.$
For simplicity of notation, we write $\mathbf{J}$ and $\mathbf{G}^{\pm}$
for $J_{-1}{\bf 1}^{c}$ and 
$G^{\pm}_{-\frac{3}{2}}{\bf 1}^{c}$, respectively.

Let $L_{c}(\mf{ns}_{2})$ be the simple
quotient vertex superalgebra of $V_{c}(\mf{ns}_{2})$.
Note that $L_{c}(\mf{ns}_{2})=\mathcal{L}_{0,0,c}$
as an $\mf{ns}_{2}$-module.
Similarly to the affine case (see \ref{Vaff}),
we write $\mathbf{X}$
 for the image of $\mathbf{X}\in V_{c}(\mf{ns}_{2})$
 in $L_{c}(\mf{ns}_{2})$.

\subsection{Heisenberg VOAs and 
lattice vertex superalgebras}

In order to formulate the $\mathcal{N}=2$ supersymmetric
 coset constructions,
we have to introduce the Heisenberg VOAs $\mathcal{F}^{\pm}$
 and the lattice conformal vertex superalgebras $V^{\pm}$
associated with certain non-degenerate integral lattices.
See \cite{Kac97V} for the details.

Let $Q^{\pm}=\mathbb{Z}\alpha^{\pm}$ be the
integral lattices with the $\mathbb{Z}$-bilinear forms defined by
$(\alpha^{\pm}|\alpha^{\pm})=\pm1$, respectively.
Define a non-degenerate $\mathbb{C}$-bilinear form
$B^{\pm}$ on $\mf{t}^{\pm}:=Q^{\pm}\otimes\mathbb{C}$ 
by linearly extending the form on $Q^{\pm}$.
Then we fix the following notations:
\begin{equation*}
(\mathcal{F}^{\pm},\ket{0}^{\pm}\!\!,\,\omega^{\pm}):=
\bigl(V_{B^{\pm}}(\mf{t}^{\pm}),\mathbf{1}_{B^{\pm}},\omega_{\mf{t}^{\pm},1}\bigr).
\end{equation*}
See \S \ref{Vaff} for the definition.

To simplify the notation, we write ${\boldsymbol \alpha}^{\pm}$ for 
$\alpha^{\pm}_{-1}\ket{0}^{\pm}\in\mathcal{F}^{\pm}$
and $L^{\pm}_{n}$ for
the linear operator $\omega^{\pm}_{(n+1)}$.
For $\lambda\in\mathbb{C}\simeq(\mf{t}^{\pm})^{*}$, we denote by $\ket{\lambda}^{\pm}$
the highest weight vector of the Heisenberg Fock module $\mathcal{F}^{\pm}_{\lambda}$.

\begin{rem}\label{HeisVOA}
The assignment ${\boldsymbol \alpha}^{-}\mapsto i{\boldsymbol \alpha}^{+}$
uniquely extends to a VOA isomorphism from $\mathcal{F}^{-}$ to $\mathcal{F}^{+}$.
We call the isomorphism class of $\mathcal{F}^{+}$ the Heisenberg VOA of rank $1$.
\end{rem}

We now recall the simple lattice vertex superalgebras
associated
with non-degenerate integral lattices $Q^{\pm}$.
Let $\mathbb{C}[Q^{\pm}]=
\bigoplus_{n\in\mathbb{Z}}\mathbb{C}e^{n\alpha^{\pm}}$
be the group algebra of $Q^{\pm}$ over $\mathbb{C}$.
It is known that there exists a simple vertex superalgebra structure
on the $\mathbb{Z}_{2}$-graded vector space
$V_{Q^{\pm}}:=\mathbb{C}[Q^{\pm}]\otimes_{\mathbb{C}}
\mathbb{C}[\alpha^{\pm}_{-n}|n>0]$
with the vacuum vector
${\bf 1}^{\pm}:=e^{0}\otimes1$.
Here the $\mathbb{Z}_{2}$-grading on $V_{Q^{\pm}}$ is defined by
\begin{equation*}
(V_{Q^{\pm}})^{\bar{i}}:=
\bigoplus_{(\beta|\beta)\in\{i\}+2\mathbb{Z}}
\mathbb{C}e^{\beta}\otimes_{\mathbb{C}}\mathbb{C}[\alpha^{\pm}_{-n}|n>0]
\end{equation*}
for $i\in\{0,1\}$.
See \cite[Theorem 5.5]{Kac97V} for the details.

We take
$\omega^{Q^{\pm}}:=e^{0}\otimes
\left(\pm\frac{1}{2}(\alpha^{\pm}_{-1})^{2}\right)$
as a confomal vector of $V_{Q^{\pm}}$
of central charge $1$.
Denote by $V^{\pm}$ the corresponding
conformal vertex superalgebra and
by $L^{Q^{\pm}}_{n}$ the linear operator
$\omega^{Q^{\pm}}_{(n+1)}$.
Note that $V^{-}$ is not a VOSA.
In what follows, we abbreviate $e^{0}\otimes u$ as $u$
and $e^{n\alpha^{\pm}}\otimes1$ as $e^{n\alpha^{\pm}}$,
where $u\in\mathbb{C}[\alpha^{\pm}_{-n}|n>0]$.

\subsection{Formulation of $\mathcal{N}=2$ coset constructions}

We give a review of the vertex superalgebraic formulation 
of the $\mathcal{N}=2$ supersymmetric 
coset construction \cite{DPYZ86}, \cite{KS89} (see also \cite{hosono1991lie})
 and its ``inverse'' coset construction \cite{FST98}.
The following reformulation is due to \cite{Ad99}.
In what follows, we always take $\kappa\in\mathbb{C}^{\times}$
and assume $\displaystyle k=-2+\frac{2}{\kappa^{2}}$.

\begin{fct}\label{coset1}
(1) There exists a unique homomorphism of conformal 
	vertex superalgebras
$\iota_{+}\colon V_{c_{k}}(\mf{ns}_{2})\otimes \mathcal{F}^{+}
\rightarrow
V_{k}(\mf{sl}_{2})\otimes V^{+}$
	such that
\begin{equation*}
\begin{array}{rcl}
	\mathbf{G}^{+}\otimes\ket{0}^{+}
	&\mapsto& \textstyle
	\kappa\mathbf{E}\otimes e^{\alpha^{+}}, \\ 
	\mathbf{G}^{-}\otimes\ket{0}^{+}
	&\mapsto&
	\kappa\mathbf{F}\otimes e^{-\alpha^{+}}, \\
	\mathbf{J}\otimes\ket{0}^{+}
	&\mapsto&
	\frac{\ \kappa^{2}}{2}\mathbf{H}\otimes{\bf 1}^{+}
	+\frac{c_{k}}{3}{\bf 1}^{k}\otimes\alpha^{+}_{-1}, \\
	{\bf 1}^{c_{k}}\otimes{\boldsymbol \alpha}^{+}	
	&\mapsto& \textstyle
	\kappa
	\bigl(\frac{1}{2}\mathbf{H}\otimes{\bf 1}^{+}-
	{\bf 1}^{k}\otimes\alpha^{+}_{-1}\bigr).
\end{array}
\end{equation*}
(2) There exists a unique homomorphism of conformal 
	vertex superalgebras
$\iota_{-}\colon V_{k}(\mf{sl}_{2})\otimes\mathcal{F}^{-}
\rightarrow
V_{c_{k}}(\mf{ns}_{2})\otimes V^{-}$ such that
\begin{equation*}
\begin{array}{rcl}
	\mathbf{E}\otimes\ket{0}^{-}
	&\mapsto& \textstyle
	\kappa^{-1}\mathbf{G}^{+}\otimes e^{-\alpha^{-}}, \\ 
	\mathbf{F}\otimes\ket{0}^{-}
	&\mapsto&
	\kappa^{-1}\mathbf{G}^{-}\otimes e^{\alpha^{-}}, \\ 
	\mathbf{H}\otimes\ket{0}^{-}
	&\mapsto&
	\frac{2}{\ \kappa^{2}}\mathbf{J}\otimes{\bf 1}^{-}
	-k{\bf 1}^{c_{k}}\otimes\alpha^{-}_{-1}, \\ 
	{\bf 1}^{k}\otimes{\boldsymbol \alpha}^{-}
	&\mapsto& \textstyle
	\kappa^{-1}
	\bigl(\mathbf{J}\otimes{\bf 1}^{-}
	-{\bf 1}^{c_{k}}\otimes\alpha^{-}_{-1}\bigr).
\end{array}
\end{equation*}
\end{fct}

In what follows, we always abbreviate
$A\otimes{\bf 1}_{W}$ and ${\bf 1}_{V}\otimes B$ in a tensor product
vertex superalgebra $V\otimes W$ as $A$ and $B$, respectively.

Let $\sigma$ be the symmetric braiding
of the category of $\mathbb{Z}_{2}$-graded vector spaces.
Then we define
$\iota_{\mf{sl}_{2}}:=(\iota_{+})_{12}\circ (\iota_{-})_{13}$
and
$\iota_{\mf{ns}_{2}}:=(\iota_{-})_{13}\circ(\iota_{+})_{12},$
where $(\iota_{+})_{12}:=\iota_{+}\otimes\id$ and
$(\iota_{-})_{13}:=(\id\otimes\sigma)\circ
(\iota_{-}\otimes\id)\circ(\id\otimes\sigma)$.

\begin{cor}\label{twist}
The mappings $\iota_{\mf{sl}_{2}}\colon V_{k}(\mf{sl}_{2})
\otimes \mathcal{F}^{+}\otimes\mathcal{F}^{-}
\rightarrow
V_{k}(\mf{sl}_{2})\otimes V^{+}\otimes V^{-}$
and
$\iota_{\mf{ns}_{2}}\colon V_{c_{k}}(\mf{ns}_{2})
\otimes \mathcal{F}^{+}\otimes\mathcal{F}^{-}
\rightarrow
V_{c_{k}}(\mf{ns}_{2})\otimes V^{+}\otimes V^{-}$
are morphisms of conformal vertex superalgebras.
\end{cor}

\section{Weight-wise admissible modules}\label{GW}

In this section we introduce the notion of weight-wise admissible modules.
The category of such modules is 
a variant of the ``category $\mathcal{O}$'' for a VOA
firstly introduced by \cite[Section 2]{dong1997regularity}.

\subsection{General notations}\label{general}

Let $(V=\bigoplus_{\Delta\in\mathbb{Q}}V_{\Delta},Y,{\bf 1},\omega)$ be a conformal vertex superalgebra.
Assume that there exists a finite-dimensional subspace
$\mf{h}$ of $V_{1}$ and a mapping 
$B\colon\mf{h}\times\mf{h}\to\mathbb{C}$ such that
the following OPE
$$Y(h;z)Y(h';w)\sim\frac{B(h,h'){\sf id}_{V}}{(z-w)^{2}}$$ 
holds for any $h,h'\in\mf{h}$.
Then $B$ is automatically symmetric and bilinear.
Then there exists a natural embedding
of $\mathcal{F}:=V_{B}(\mf{h})$ into $V$.
In what follows, we fix such a pair $(V,\mf{h})$.

\begin{dfn}\label{WWA}
A weak $V$-module $M$ is a \emph{weight-wise admissible $(V,\mf{h})$-module} if
\begin{enumerate}
\item We have the decomposition
\begin{equation}\label{decomp}
M=\bigoplus_{\lambda\in\mf{h}^{*}}M(\lambda),
\end{equation}
where 
$M(\lambda):=
\{m\in M\,|\,h_{(0)}m=\lambda(h)m\text{ for any }h\in\mf{h}\}.$
\item For any $\lambda\in\mf{h}^{*}$, we have the generalized eigenspace decomposition
$$M(\lambda)=\bigoplus_{\Delta\in\mathbb{C}}M(\Delta,\lambda)$$
with respect to $L^{M}_{0}=\omega_{(1)}^{M}$
such that each component is finite-dimensional.
\item
For any $\lambda\in\mf{h}^{*}$, the set $\bigl\{{\rm Re}(\Delta)\,\bigl|\,M(\Delta,\lambda)\neq\{0\}\bigr.\bigr\}$
is lowerly bounded.
\end{enumerate}
\end{dfn}

For a weight-wise admissible $(V,\mf{h})$-module $M$, we define
$$P(M):=\bigl\{(\Delta,\lambda)\in\mathbb{C}\times\mf{h}^{*}
\,\bigl|\,M(\Delta,\lambda)\neq\{0\}\bigr.\bigr\}.$$

\begin{dfn}
Suppose that $P(M)$ is discrete.
Then we define the formal character of $M$ by
\begin{equation*}
\ch(M):=\sum_{(h,\lambda)\in P(M)}
\bigl({\sf dim} M(h,\lambda)\bigr)q^{h}\mathbf{e}^{\lambda},
\end{equation*}
where $q$ and $\mathbf{e}$ are formal variables.
\end{dfn}

The following lemma describes the $\mathcal{F}$-module structure
of a weight-wise admissible $(V,\mf{h})$-module.

\begin{lem}\label{LBDD}
Let $M$ be a weight-wise admissible $(V,\mf{h})$-module.
Assume that the mapping $B\colon\mf{h}\times\mf{h}\to\mathbb{C}$
is non-degenerate.
Then the linear mapping
$$\bigoplus_{\lambda\in\mf{h}^{*}}
\Hom_{\mathcal{F}}\bigl(\mathcal{F}_{\lambda},M\bigr)\otimes\mathcal{F}_{\lambda}
\rightarrow M;\,f\otimes v\mapsto f(v)$$
gives a weak $\mathcal{F}$-module isomorphism,
where $\mathcal{F}_{\lambda}:=
\Ind_{\widehat{\mf{h}}_{B}^{\geq0}}^{\widehat{\mf{h}}_{B}}\mathbb{C}_{\lambda}$
is the Heisenberg Fock module of highest weight $\lambda\in\mf{h}^{*}$.
\end{lem}

\begin{proof}
The injectivity of the mapping follows from the decomposition (\ref{decomp})
and the simplicity of $\mathcal{F}_{\lambda}$.
By the condition (3) in Definition \ref{WWA}, 
the $\widehat{\mf{h}}_{B}$-module
$M(\lambda)$ satisfies the condition $\mf{C}_{1}$
in \cite[Section 1.7]{frenkel1989vertex}
(equivalently, $M(\lambda)$ is an admissible $\mathcal{F}$-module 
in the sense of \cite[Definition 2.3]{dong1997regularity}).
Then the surjectivity follows from \cite[Theorem 1.7.3]{frenkel1989vertex}.
\end{proof}

Denote by $\mathcal{C}_{V}=\mathcal{C}_{(V,\mf{h})}$
the full subcategory of $V$-{\sf Mod} whose objects are
 weight-wise admissible $(V,\mf{h})$-modules.
Since weight-wise admissible $(V,\mf{h})$-modules
are closed under finite sums, kernels and cokernels,
the category $\mathcal{C}_{V}$ is abelian.

\begin{dfn}\label{block}
Assume that $(V,Y)$ is a weight-wise admissible
$(V,\mf{h})$-module.
Let $Q_{V}$ be the additive subgroup of $\mathbb{C}\times\mf{h}^{*}$
generated by $P(V)$.
For $\Delta\in\mathbb{C}$ and $\lambda\in\mf{h}^{*}$,
we denote 
the full subcategory of 
$\mathcal{C}_{V}$ whose objects satisfy
$P(M)\subset\overline{(\Delta,\lambda)}:=\{(\Delta,\lambda)\}+Q_{V}$
by 
$\mathcal{C}_{V}^{\overline{(\Delta,\lambda)}}$.
We call this full subcategory the \emph{$\overline{(\Delta,\lambda)}$-block}
of $\mathcal{C}_{V}$.
\end{dfn}

\begin{rem}
We note that there are no non-trivial homomorphisms and extensions between modules 
from distinct blocks of $\mathcal{C}_{V}$.
\end{rem}

\subsection{Our setting}\label{Setting}

In what follows, $V_{\mf{sl}_{2}}$ stands for
the VOA $V_{k}(\mf{sl}_{2})$ or $L_{k}(\mf{sl}_{2})$,
and $V_{\mf{ns}_{2}}$ stands for
the VOSA $V_{c_{k}}(\mf{ns}_{2})$ or $L_{c_{k}}(\mf{ns}_{2})$.
Throughout this paper,
we fix the following subspaces:
\begin{equation*}
\mf{h}\bigl(V_{\mf{sl}_{2}}\bigr)
:=\mathbb{C}\mathbf{H},\ 
\mf{h}\bigl(V_{\mf{ns}_{2}}\bigr)
:=\mathbb{C}\mathbf{J},\ 
\mf{h}(\mathcal{F}^{\pm})
:=\mathbb{C}{\boldsymbol \alpha}^{\pm},\ 
\mf{h}(V^{\pm})
:=\mathbb{C}\alpha_{-1}^{\pm}.
\end{equation*}
Let $\mathbf{H}^{*}, \mathbf{J}^{*}, ({\boldsymbol \alpha}^{\pm})^{*}$
and $(\alpha_{-1}^{\pm})^{*}$ denote the dual basis
for $\mathbf{H}, \mathbf{J}, {\boldsymbol \alpha}^{\pm}$
and $\alpha_{-1}^{\pm}$, respectively.
For the tensor vertex superalgebra $V\otimes W$,
we put 
$$\mf{h}(V\otimes W):=\mf{h}(V)\otimes{\bf 1}_{W}+{\bf 1}_{V}\otimes\mf{h}(W).$$

Let $V$ be one of the above conformal vertex superalgebras.
Then it is easy to verify that
the adjoint module $(V,Y)$ is a weight-wise admissible $\bigl(V,\mf{h}(V)\bigr)$-module.
By fixing appropriate bases, we obtain the following identifications:
\begin{equation*}
Q_{V_{\mf{sl}_{2}}}\simeq\mathbb{Z}^{2},\ 
Q_{V_{\mf{ns}_{2}}}\simeq
\left\{\left.\left(a+\frac{b}{2},b\right)\,\right|a,b\in\mathbb{Z}\right\},
\end{equation*}
\begin{equation*}
Q_{\mathcal{F}^{\pm}}\simeq\mathbb{Z}\times\{0\},\ 
Q_{V^{\pm}}\simeq
\left\{\left.\left(a+\frac{b}{2},b\right)\,\right|a,b\in\mathbb{Z}\right\}.
\end{equation*}

\subsection{Examples}

\subsubsection{Highest weight modules over $\widehat{\mf{sl}}_{2}$ and $\mf{ns}_{2}$}

The Verma module $M_{j,k}$ of $\widehat{\mf{sl}}_{2}$
and its quotient modules lie in the block 
$\mathcal{C}_{V_{k}(\mf{sl}_{2})}^{\overline{(\Delta_{j},j)}}$,
where $\Delta_{j}:=\frac{j(j+1)}{k+2}$.
On the other hand, the Verma module $\mathcal{M}_{h,j,c}$ of $\mf{ns}_{2}$
and its quotients lie in the block $\mathcal{C}_{V_{c}(\mf{ns}_{2})}^{\overline{(h,j)}}$.

\subsubsection{Relaxed highest weight modules over $\widehat{\mf{sl}}_{2}$}

Since the block $\mathcal{C}_{V_{k}(\mf{sl}_{2})}^{\overline{(h,j)}}$
contains no highest weight modules for generic $h,j\in\mathbb{C}$,
we introduce a generalization of usual highest weight modules
as the appropriate counterpart of 
the highest weight $\mf{ns}_{2}$-modules.

Let $U_{0}:=\{u\in U(\mf{sl}_{2})\mid[H,u]=0\}$ be 
the subalgebra of $U(\mf{sl}_{2})$.
It is clear that $H$ and the Casimir element 
$\Omega\in U(\mf{sl}_{2})$
freely generate $U_{0}$ as a unital commutative 
$\mathbb{C}$-algebra.
For $h,j\in\mathbb{C}$, let $\mathbb{C}_{h,j}$ be
the 1-dimensional $U_{0}$-module defined 
by $\Omega.1:=2(k+2)h$ and $H.1:=2j$.
We define the action of 
$\mf{sl}_{2}\otimes \mathbb{C}[t]
\oplus\mathbb{C}K$ on the induced $\mf{sl}_{2}$-module
$\Ind_{U_{0}}^{U(\mf{sl}_{2})}\mathbb{C}_{h,j,k}$ by
$X_{n}\mapsto X\delta_{n,0},\,K\mapsto k\id$ for $X\in\mf{sl}_{2}$ and
$n\geq0$.
Denote this $\mf{sl}_{2}\otimes \mathbb{C}[t]
\oplus\mathbb{C}K$-module by $\bar{R}_{h,j,k}$.
Then, we define the relaxed Verma module (cf. \cite{FST98}) 
as the induced module
$R_{h,j,k}:=\Ind_{\mf{sl}_{2}\otimes \mathbb{C}[t]
\oplus\mathbb{C}K}^{\widehat{\mf{sl}}_{2}}\bar{R}_{h,j,k}.$
We write $\ket{h,j,k}$ for the canonical generator 
of $R_{h,j,k}$.

By an easy computaion, we have 
$L^{\mf{sl}_{2}}_{0}\ket{h,j,k}=h\ket{h,j,k}$.
Hence $R_{h,j,k}$ lies in $\mathcal{C}_{V_{k}(\mf{sl}_{2})}^{\overline{(h,j)}}$.
Since the sum of all proper submodules of $R_{h,j,k}$
 which intersect trivially with 
$\mathbb{C}\ket{h,j,k}$ gives the maximum
proper submodule, 
there exists a unique irreducible quotient 
$L_{h,j,k}$ of $R_{h,j,k}$.

\begin{rem}\label{divergent}
Since $R_{h,j,k}$ is a free $U(\mf{sl}_{2}\otimes t^{-1}\mathbb{C}[t^{-1}])$-module
with base 
$$\{E_{0}^{n}\ket{h,j,k},\ket{h,j,k},F_{0}^{n}\ket{h,j,k}|\,n>0\,\},$$ 
we have
$$\ch(R_{h,j,k})(q,x)
=q^{h}x^{j}\sum_{n\in\mathbb{Z}}\frac{x^{n}}
{(q;q)^{3}_{\infty}},$$
where $x:=\mathbf{e}^{\frac{1}{2}\mathbf{H}^{*}}$ (see \S\ref{Setting}) and $(a;q)_{\infty}:=\prod_{i\geq0}(1-aq^{i})$.
Note that the formal series 
$\ch(\mathcal{L}_{h,j,c})(q,x)$ is absolutely convergent 
as a function in two variables $q$ and $x$ in the region
$\mathbb{A}:=
\left\{(q,x)\in\mathbb{C}^{2}
\left|\,0<|q|<1,\ 
|q|^{\frac{1}{2}}<|x|<|q|^{-\frac{1}{2}}
\right.\right\},$
while
$\ch(L_{h,j,k})(q,x)$ 
has no convergent region in general.
\end{rem}

In \cite{futorny1996irreducible},
V. Futorny introduces a generalization of Verma modules over 
the affine Kac-Moody algebra 
$\widehat{\mf{sl}}_{2}\rtimes\mathbb{C}D$.
It is clear that the modules are isomorphic to 
relaxed Verma modules with spectral flow twists
(see Appendix \ref{SFA}) as $\widehat{\mf{sl}}_{2}$-modules.
Consequently, the next fact follows from 
\cite[Theorem 6.3]{futorny1996irreducible}, the classification
of irreducible weight modules over $\widehat{\mf{sl}}_{2}\rtimes\mathbb{C}D$.

\begin{fct}\label{simple}
Assume $k\neq0$ and let $L$ be a simple object in $\mathcal{C}_{V_{k}(\mf{sl}_{2})}$.
Then, there exist $h,j\in\mathbb{C}$ and $\theta\in\mathbb{Z}$ such that
$L\simeq L_{h,j,k}^{\theta}.$
\end{fct}

\section{Equivalence between module categories}

In this section we construct functors which
establish categorical equivalences, and state the main result.
Throughout this section, we fix $h,j\in\mathbb{C}$.

\subsection{Functors $\Omega^{\pm}_{j}$}

In this subsection we introduce functors between
$V_{k}(\mf{sl}_{2})$-{\sf Mod} and $V_{c_{k}}(\mf{ns}_{2})$-{\sf Mod}.
The restrictions of these functors
give categorical equivalences
between blocks of the category of weight-wise admissible modules.

\begin{dfn}
We define the functors
$$
\widetilde{\Omega}^{+}_{j}:=\Hom_{\mathcal{F}^{+}}\left(\mathcal{F}^{+}_{\kappa j},
\iota_{+}^{*}(-\otimes V^{+})\right)\colon
V_{k}(\mf{sl}_{2})\text{\sf-Mod}\rightarrow
V_{c_{k}}(\mf{ns}_{2})\text{\sf-Mod},$$
$$\widetilde{\Omega}^{-}_{j}:=
\Hom_{\mathcal{F}^{-}}\left(\mathcal{F}^{-}_{\kappa j},
\iota_{-}^{*}(-\otimes V^{-})\right)
\colon
V_{c_{k}}(\mf{ns}_{2})\text{\sf-Mod}
\rightarrow
V_{k}(\mf{sl}_{2})\text{\sf-Mod}.$$
\end{dfn}

From now on, we always suppose
\begin{equation}\label{assump}
M\in\Obj\left(\mathcal{C}_{V_{k}(\mf{sl}_{2})}^{\overline{(h,j)}}\right)
\text{ and }\ 
N\in\Obj\left(\mathcal{C}_{V_{c_{k}}(\mf{ns}_{2})}
^{\overline{(h-\frac{j^{2}}{k+2},\frac{2j}{k+2})}}\right).
\end{equation}

\begin{lem} We have
\begin{equation*}
\widetilde{\Omega}^{+}_{j}(M)
\in\Obj\left(\mathcal{C}_{V_{c_{k}}(\mf{ns}_{2})}
^{\overline{(h-\frac{j^{2}}{k+2},\frac{2j}{k+2})}}\right)
\text{ and }\ 
\widetilde{\Omega}^{-}_{j}(N)
\in\Obj\left(
\mathcal{C}_{V_{k}(\mf{sl}_{2})}^{\overline{(h,j)}}\right).
\end{equation*}
\end{lem}

\begin{proof}
Since $\iota_{\pm}$ preserve the subspaces $\mf{h}$
of the corresponding vertex superalgebras,
the conditions in Definition \ref{WWA}
for $\Omega^{+}_{j}(M)$ and $\Omega^{-}_{j}(N)$
 inherit from those for $M$, $N$ and $V^{\pm}$.
By easy computations,
it follows that these modules lie in the above blocks.
\end{proof}

In what follows, we write $\Omega^{+}_{j}$ and $\Omega^{-}_{j}$
for the restrictions of the functors 
$\widetilde{\Omega}^{+}_{j}$ and $\widetilde{\Omega}^{-}_{j}$
to the full subcategories $\mathcal{C}_{V_{k}(\mf{sl}_{2})}^{\overline{(h,j)}}$ and
$\mathcal{C}_{V_{c_{k}}(\mf{ns}_{2})}
^{\overline{(h-\frac{j^{2}}{k+2},\frac{2j}{k+2})}}$, respectively.
To simplify notation, we set
$$\mathcal{F}_{(n,m)}^{+}
:=\mathcal{F}
_{\kappa(j+\frac{k}{2}n-m)}^{+},\ \ 
\mathcal{F}_{(n,m)}^{-}
:=\mathcal{F}_{\kappa
(j+\frac{k}{2}n+\frac{k+2}{2}m)}^{-}$$
for $n,m\in\mathbb{Z}$.
The functors $\Omega^{\pm}_{j}$
are motivated by the following isomorphisms.

\begin{lem}\label{vacc}
(1) The linear mapping 
\begin{equation}\label{vac1}
\bigoplus_{n\in\mathbb{Z}}
\Omega^{+}_{j-n}(M)\otimes \mathcal{F}^{+}_{(0,n)}
\overset{\simeq}{\longrightarrow}
\iota_{+}^{*}(M\otimes V^{+});\,f\otimes v\mapsto f(v)
\end{equation}
is a $V_{c_{k}}(\mf{ns}_{2})\otimes \mathcal{F}^{+}$-module isomorphism.

(2) The linear mapping
\begin{equation}\label{vac2}
\bigoplus_{m\in\mathbb{Z}}
\Omega^{-}_{j+\frac{k+2}{2}m}(N)\otimes\mathcal{F}^{-}_{(0,m)}
\overset{\simeq}{\longrightarrow}
\iota_{-}^{*}(N\otimes V^{-})
;\,g\otimes w\mapsto g(w)
\end{equation}
is a  $V_{k}(\mf{sl}_{2})\otimes\mathcal{F}^{-}$-module isomorphism.
\end{lem}

Since the proof is similar to that of Lemma \ref{LBDD},
we omit it.

\subsection{Statement of main result}

Our main result in this paper is as follows:

\begin{theom}\label{main}
The two functors
$$\Omega^{+}_{j}\colon\mathcal{C}_{V_{k}(\mf{sl}_{2})}^{\overline{(h,j)}}
\to\mathcal{C}_{V_{c_{k}}(\mf{ns}_{2})}
^{\overline{(h-\frac{j^{2}}{k+2},\frac{2j}{k+2})}},\ 
\Omega^{-}_{j}\colon
\mathcal{C}_{V_{c_{k}}(\mf{ns}_{2})}^{\overline{(h-\frac{j^{2}}{k+2},\frac{2j}{k+2})}}
\to
\mathcal{C}_{V_{k}(\mf{sl}_{2})}^{\overline{(h,j)}}$$
are mutually quasi-inverse to each other.
In particular, these functors give categorical equivalences
as $\mathbb{C}$-linear abelian cateogries.
\end{theom}

We prove it in the next section.

\section{Proof of main result}

In this section we give the proof of Theorem \ref{main} by steps.
Throughout this section, we fix $h,j\in\mathbb{C}$
and suppose (\ref{assump}).

\subsection{Branching rules with respect to 
$\iota_{\mf{sl}_{2}}^{*}$ and $\iota_{\mf{ns}_{2}}^{*}$}

\begin{lem}\label{twvacc}
(1) As a $V_{k}(\mf{sl}_{2})\otimes \mathcal{F}^{+}\otimes\mathcal{F}^{-}$
-module, we have
\begin{equation*}
\bigoplus_{n,m\in\mathbb{Z}}
\Omega^{-}_{j+\frac{k}{2}n+\frac{k+2}{2}m}\circ
\Omega^{+}_{j-n}(M)\otimes \mathcal{F}^{+}_{(0,n)}
\otimes\mathcal{F}^{-}_{(n,m)}
\overset{\simeq}{\longrightarrow}
\iota_{\mf{sl}_{2}}^{*}(M\otimes V^{+}\otimes V^{-}).
\end{equation*}

(2) As a $V_{c_{k}}(\mf{ns}_{2})\otimes \mathcal{F}^{+}\otimes\mathcal{F}^{-}$
-module, we have
\begin{equation*}
\bigoplus_{m\in\mathbb{Z}}
\Omega^{+}_{j+\frac{k}{2}m-n}\circ\Omega^{-}_{j+\frac{k+2}{2}m}(N)
\otimes \mathcal{F}^{+}_{(m,n)}\otimes\mathcal{F}^{-}_{(0,m)}
\overset{\simeq}{\longrightarrow}
\iota_{\mf{ns}_{2}}^{*}(N\otimes V^{+}\otimes V^{-}).
\end{equation*}
\end{lem}

\begin{proof}
Since
$\iota_{\mf{sl}_{2}}^{*}(M\otimes V^{+}\otimes V^{-})
=(\iota_{-})_{13}^{*}\circ(\iota_{+})_{12}^{*}(M\otimes V^{+}\otimes V^{-}),$
there exists a natural $V_{c_{k}}(\mf{ns}_{2})$-module isomorphism
$$\Omega^{-}_{j+\frac{k}{2}n+\frac{k+2}{2}m}\circ
\Omega^{+}_{j-n}(M)\simeq
\Hom_{\mathcal{F}^{+}\otimes\mathcal{F}^{-}}\bigl(\mathcal{F}^{+}_{(0,n)}
\otimes\mathcal{F}^{-}_{(n,m)},\iota_{\mf{sl}_{2}}^{*}(M\otimes V^{+}\otimes V^{-})\bigr).$$
Through the natural isomorphism, we define 
a $V_{k}(\mf{sl}_{2})\otimes \mathcal{F}^{+}\otimes\mathcal{F}^{-}$
-module homomorphism
$$\bigoplus_{n,m\in\mathbb{Z}}
\Omega^{-}_{j+\frac{k}{2}n+\frac{k+2}{2}m}\circ
\Omega^{+}_{j-n}(M)\otimes \mathcal{F}^{+}_{(0,n)}
\otimes\mathcal{F}^{-}_{(n,m)}\to
\iota_{\mf{sl}_{2}}^{*}(M\otimes V^{+}\otimes V^{-})$$
 by $f\otimes v^{+}\otimes v^{-}\mapsto f(v^{+}\otimes v^{-})$.
The bijectivity of this mapping is proved in a similar way as Lemma \ref{LBDD}.
The proof of (2) is same as that of (1).
\end{proof}

\subsection{Calculation of formal characters}

Let $\theta\in\mathbb{Z}$.
We use the following notations:
$$\mathcal{C}_{V_{k}(\mf{sl}_{2})}
^{\overline{(h,j;\,\theta)}}:=
\mathcal{C}_{V_{k}(\mf{sl}_{2})}
^{\overline{(h+j\theta+\frac{k\theta^{2}}{4},j+\frac{k\theta}{2})}},\ \ 
\mathcal{C}_{V_{c_{k}}(\mf{ns}_{2})}
^{\overline{(h,j;\,\theta)}}:=
\mathcal{C}_{V_{c_{k}}(\mf{ns}_{2})}
^{\overline{(h-\frac{(j-\theta)^{2}}{k+2},\frac{2(j-\theta)}{k+2})}}.$$

\begin{lem}\label{twchar}
For $\theta\in\mathbb{Z}$, the twisted modules
$M^{\theta}$ and $N^{\theta}$
(see Appendix \ref{SFA} for the definition) lie in 
$\mathcal{C}_{V_{k}(\mf{sl}_{2})}
^{\overline{(h,j;\,\theta)}}$ and 
$\mathcal{C}_{V_{c_{k}}(\mf{ns}_{2})}
^{\overline{(h,j;\,\theta)}}$, respectively.
Moreover, we have
$\ch(M^{\theta})(q,x)=q^{\frac{\ k\theta^{2}}{4}}
x^{\frac{k\theta}{2}}\ch(M)(q,xq^{\theta})$ and
$\ch(N^{\theta})(q,x)=q^{\frac{\ k\theta^{2}}{2(k+2)}}
x^{\frac{k\theta}{k+2}}\ch(N)(q,xq^{\theta}).$
\end{lem}

\begin{proof}
As a corollary of Lemma \ref{sptw}, 
we see that the linear operator $(\omega^{\mf{sl}_{2}})^{M^{\theta}}_{(1)}$
coincides with $L^{\mf{sl}_{2}}_{0}+\frac{\theta}{2}H_{0}
+\frac{\ \theta^{2}}{4}K$ as an element of $\End(M)$.
For $(h',j')\in\mathbb{C}^{2}$, we put 
$(\tilde{h'},\tilde{j}')
:=(h'+j'\theta+
\frac{\ k\theta^{2}}{4},j'+\frac{k\theta}{2})
\in\mathbb{C}^{2}$.
Since $(\tilde{h}',\tilde{j}')=(\tilde{h}'',\tilde{j}'')$
holds if and only if $(h',j')=(h'',j'')$,
we have ${\sf dim} M_{h',j'}={\sf dim}(M^{\theta})
_{\tilde{h}',\tilde{j}'}$.
Hence $M^{\theta}$ inherits the conditions
of weight-wise admissiblity and the former equality holds.

Since we have
${\sf dim} N_{h,\lambda}={\sf dim}(N^{\theta})
_{h+\theta\lambda+\frac{\ k\theta^{2}}{2(k+2)},
\lambda+\frac{k\theta}{k+2}}$
by Lemma \ref{sptw},
the statements for $N^{\theta}$ are proved in the same way.
\end{proof}

As a corollary, we have the following categorical isomorphisms, 
which play important roles in the next section. 
\begin{cor}
The restrictions of the functors $\Delta^{\theta}_{\mf{sl}_{2}}$
 and $\Delta^{\theta}_{\mf{ns}_{2}}$
to the full subcategories
$\mathcal{C}_{V_{k}(\mf{sl}_{2})}^{\overline{(h,j;\,0)}}$
and $\mathcal{C}_{V_{c}(\mf{ns}_{2})}
^{\overline{(h,j;\,0)}}$ give categorical isomorphisms
$\mathcal{C}_{V_{k}(\mf{sl}_{2})}^{\overline{(h,j;\,0)}}
\overset{\simeq}{\longrightarrow}
\mathcal{C}_{V_{k}(\mf{sl}_{2})}
^{\overline{(h,j;\,\theta)}}$
and
$\mathcal{C}_{V_{c}(\mf{ns}_{2})}
^{\overline{(h,j;\,0)}}
\overset{\simeq}{\longrightarrow}
\mathcal{C}_{V_{c}(\mf{ns}_{2})}
^{\overline{(h,j;\,\theta)}}$,
respectively.
\end{cor}

\begin{lem}\label{character}
(1) As a formal character of $V_{k}(\mf{sl}_{2})
\otimes \mathcal{F}^{+}\otimes\mathcal{F}^{-}$-module,
we have
\begin{equation*}
\ch\left(\bigoplus_{n,m\in\mathbb{Z}}M^{\theta+n+m}
\otimes\mathcal{F}_{(\theta,n)}^{+}
\otimes\mathcal{F}_{(\theta+n,m)}^{-}\right)
=\ch\Bigl(
\iota_{\mf{sl}_{2}}^{*}(M^{\theta}
\otimes V^{+}\otimes V^{-})\Bigr).
\end{equation*}
(2) As a formal character of $V_{c_{k}}(\mf{ns}_{2})
\otimes \mathcal{F}^{+}\otimes\mathcal{F}^{-}$-module,
we have
\begin{equation*}
\ch\left(\bigoplus_{n,m\in\mathbb{Z}}N^{\theta+n+m}
\otimes\mathcal{F}_{(\theta+m,n)}^{+}
\otimes\mathcal{F}_{(\theta,m)}^{-}\right)
=\ch\Bigl(
\iota_{\mf{ns}_{2}}^{*}(N^{\theta}
\otimes V^{+}\otimes V^{-})\Bigr).
\end{equation*}
\end{lem}

\begin{proof}
Since $M$ and $N$ lie in blocks, there exist
$f(q,x),\,g(q,x)\in\mathbb{Z}_{\geq0}[\![q^{\pm1},x^{\pm1}]\!]$
such that
$\ch(M)(q,x)=q^{h}x^{j}f(q,x)$ and
$\ch(N)(q,x)=q^{h-\frac{j^{2}}{k+2}}x^{\frac{2j}{k+2}}g(q,x).$

(1) For simplicity,
we set the variables by 
$$x:=\mathbf{e}^{\frac{1}{2}\mathbf{H}^{*}},\ 
y:=\mathbf{e}^{\kappa({\boldsymbol \alpha}^{+})^{*}},\ 
z:=\mathbf{e}^{\kappa({\boldsymbol \alpha}^{-})^{*}}$$
for weight-wise admissible $V_{k}(\mf{sl}_{2})
\otimes \mathcal{F}^{+}\otimes\mathcal{F}^{-}$-modules (see \S\ref{Setting}).
Then, the left- and right-hand sides (LHS and RHS for short)
are calculated as
$$\text{(LHS)}
=
\displaystyle
\frac{q^{h+j\theta+\frac{k\theta^{2}}{4}}
(xyz)^{j+\frac{k\theta}{2}}}
{(q;q)_{\infty}^{2}}
\sum_{n,m\in\mathbb{Z}}
q^{\Delta_{n,m}}
(xz)^{\frac{k}{2}(n+m)}
y^{-n}z^{m}
f(q,xq^{\theta+n+m}),$$
\begin{equation*}
\begin{array}{ll}
\hspace{-4mm}\text{(RHS)}
&\hspace{-2mm}
=\ch\Bigl(M^{\theta}
\otimes V^{+}\otimes V^{-}\Bigr)
(q,xyz,y^{-1}(xz)^{\frac{k}{2}},
z^{-1}(xz)^{-\frac{k}{2}})\\
&\hspace{-2mm}
=\displaystyle
\frac{q^{h+j\theta+\frac{k\theta^{2}}{4}}
(xyz)^{j+\frac{k\theta}{2}}}
{(q;q)_{\infty}^{2}}
\sum_{n,m\in\mathbb{Z}}
q^{\Delta_{n,m}}
(xz)^{\frac{k}{2}(n+m)}
y^{-n}z^{m}
f(q,xyzq^{\theta})
\end{array}
\end{equation*}
where $\Delta_{n,m}:=\frac{(n+m)(n-m)}{2}$.
So we can reduce them to
\begin{equation*}
\begin{array}{ccl}
\text{(LHS)}^{\prime}
&:=& \displaystyle
\sum_{n,m\in\mathbb{Z}}
q^{\Delta_{n,m}}
(xz)^{\frac{k}{2}(n+m)}
y^{-n}z^{m}f(q,xq^{\theta+n+m}), \\
\text{(RHS)}^{\prime}
&:=& \displaystyle
\sum_{n,m\in\mathbb{Z}}
q^{\Delta_{n,m}}
(xz)^{\frac{k}{2}(n+m)}
y^{-n}z^{m}
f(q,xyzq^{\theta}).
\end{array}
\end{equation*}
Here we put $\ell:=n+m, A:=xq^{\theta+\ell},
B:=xyzq^{\theta}$ and 
$C:=y^{-1}(xz)^{\frac{k}{2}}$,
then 
\begin{equation*}
\begin{array}{ccl}
\text{(LHS)}^{\prime}
&=& \displaystyle
\sum_{\ell\in\mathbb{Z}}q^{\frac{\ \ell^{2}}{2}}C^{\ell}
\sum_{m\in\mathbb{Z}}\left(\frac{B}{A}\right)^{m}f(q,A),\\
\text{(RHS)}^{\prime}
&=& \displaystyle
\sum_{l\in\mathbb{Z}}q^{\frac{\ \ell^{2}}{2}}C^{\ell}
\sum_{m\in\mathbb{Z}}\left(\frac{B}{A}\right)^{m}f(q,B).
\end{array}
\end{equation*}
Since $A$ and $B$ are independent of $m\in\mathbb{Z}$,
and $f(q,X)$ is an element of $\mathbb{Z}_{\geq0}[\![q^{\pm1},X^{\pm1}]\!]$,
we have an equality
\begin{equation*}
\sum_{m\in\mathbb{Z}}\left(\frac{B}{A}\right)^{m}f(q,A)
=
\sum_{m\in\mathbb{Z}}\left(\frac{B}{A}\right)^{m}f(q,B)
\end{equation*}
in $\mathbb{Z}_{\geq0}[\![q^{\pm1},x^{\pm1}
,y^{\pm1},z^{\pm1}]\!]$.
We thus get (LHS)\,=\,(RHS).

(2) The proof is similar to (1) and we omit it.
\end{proof}

\subsection{Twisted embedding}\label{BRC}

In this subsection we give an explicit form of the `branching' rule in Lemma \ref{twvacc},
which is a generalization of a result in \cite{FST98}.

By the assumption (\ref{assump}), 
we have the eigenspace decompositions
$M=\bigoplus_{a\in\mathbb{Z}}M(j+a)$
and
$N=\bigoplus_{a\in\mathbb{Z}}N(\kappa^{2}j+a).$
From now on, we always denote the decompositions of $v\in M$ and $w\in N$
by 
$v=\sum_{a}v_{a}\in\bigoplus_{a\in\mathbb{Z}}M(j+a)$
and
$w=\sum_{a}w_{a}\in\bigoplus_{a\in\mathbb{Z}}N(\kappa^{2}j+a),$
 respectively.

We introduce the following two operators:
\begin{flalign*}
&\mathcal{H}:=\sum_{n>0}\frac{1}{2}H_{n}\otimes
\left(\frac{\alpha^{+}_{-n}}{n}\otimes\id_{V^{-}}
-\id_{V^{+}}\otimes\frac{\alpha^{-}_{-n}}{n}\right)
\in\End(M\otimes V^{+}\otimes V^{-}), \\
&\mathcal{J}:=\sum_{n>0}J_{n}\otimes
\left(\frac{\alpha^{+}_{-n}}{n}\otimes\id_{V^{-}}
-\id_{V^{+}}\otimes\frac{\alpha^{-}_{-n}}{n}\right)
\in\End(N\otimes V^{+}\otimes V^{-}).
\end{flalign*}

\begin{lem}
The infinite sums
$e^{\mathcal{H}}:=\exp(\mathcal{H})$
and
$e^{\mathcal{J}}:=\exp(\mathcal{J})$
give well-defined operators on
$M\otimes V^{+}\otimes V^{-}$ and 
$N\otimes V^{+}\otimes V^{-}$, respectively.
\end{lem}

\begin{proof}
By the condition (3) in Definition \ref{WWA},
we have $\mathcal{H}^{m}(M(j+a)\otimes V^{+}\otimes V^{-})=\{0\}$ for any $a\in\mathbb{Z}$
and $m\gg0$. Therefore $e^{\mathcal{H}}$ is well-defined.
The well-definedness of $e^{\mathcal{J}}$ is proved in the same way.
\end{proof}

\begin{lem}\label{key}
(1) There exists a unique $V_{k}(\mf{sl}_{2})
\otimes \mathcal{F}^{+}\otimes\mathcal{F}^{-}$-module embedding
$\mathscr{F}_{M}\colon
	M\otimes\mathcal{F}_{(0,0)}^{+}
	\otimes\mathcal{F}_{(0,0)}^{-}
	\rightarrow
	\iota_{\mf{sl}_{2}}^{*}(M
	\otimes V^{+}\otimes V^{-})$
such that
\begin{equation}\label{em1}
	v\otimes\ket{(0,0)}^{+}
	\otimes\ket{(0,0)}^{-}\mapsto
	e^{\mathcal{H}}
	\left(
	\sum_{a\in\mathbb{Z}}v_{a}\otimes e^{a\alpha^{+}}
	\otimes e^{-a\alpha^{-}}\right)
\end{equation}
for any $v\in M$.

(2) There exists a unique $V_{c_{k}}(\mf{ns}_{2})
\otimes \mathcal{F}^{+}\otimes\mathcal{F}^{-}$-module
embedding
$\mathscr{G}_{N}\colon
	N\otimes\mathcal{F}_{(0,0)}^{+}
	\otimes\mathcal{F}_{(0,0)}^{-} \rightarrow
	\iota_{\mf{ns}_{2}}^{*}(N
	\otimes V^{+}\otimes V^{-})$
such that
\begin{equation}\label{em2}
	w\otimes\ket{(0,0)}^{+}\otimes\ket{(0,0)}^{-}
	\mapsto
	e^{\mathcal{J}}
	\left(
	\sum_{a\in\mathbb{Z}}w_{a}\otimes e^{a\alpha^{+}}
	\otimes e^{-a\alpha^{-}}\right)
\end{equation}
for any $w\in N$.
\end{lem}

\begin{proof}
(1) Let $a\in\mathbb{Z}$ and put $e(a):=
e^{a\alpha^{+}}\otimes e^{-a\alpha^{-}}$.
Denote by $\widetilde{X}_{n}$ the linear operator
$\iota_{\mf{sl}_{2}}(\mathbf{X})_{(n)}$
on $\iota_{\mf{sl}_{2}}^{*}(M\otimes V^{+}\otimes V^{-})$
, where $\mathbf{X}$
is one of the vectors 
$\mathbf{H},\mathbf{E},\mathbf{F}$ and ${\boldsymbol \alpha}^{\pm}$.
By computations, we have
\begin{flalign*}
\widetilde{H}_{n}e^{\mathcal{H}}
\bigl(v_{a}\otimes e(a)\bigr)
&=
e^{\mathcal{H}}\bigl(
H_{n}v_{a}\otimes e(a)\bigr),
\\
\widetilde{E}_{n}e^{\mathcal{H}}
\bigl(v_{a}\otimes e(a)\bigr)
&=
e^{\mathcal{H}}\bigl(
E_{n}v_{a}\otimes e(a+1)\bigr),
\\
\widetilde{F}_{n}e^{\mathcal{H}}
\bigl(v_{a}\otimes e(a)\bigr)
&=
e^{\mathcal{H}}\bigl(
F_{n}v_{a}\otimes e(a-1)
\bigr)
\end{flalign*}
for $n\in\mathbb{Z}$.
Hence (\ref{em1}) defines 
a $V_{k}(\mf{sl}_{2})$-module homomorphism
\begin{equation}\label{4.2.1}
	M
	\otimes\mathbb{C}\ket{(0,0)}^{+}
	\otimes\mathbb{C}\ket{(0,0)}^{-}
	\rightarrow
	\iota_{\mf{sl}_{2}}^{*}(M
	\otimes V^{+}\otimes V^{-}).
\end{equation}
The injectivity of this mapping follows
from the bijectivity of the operator $e^{\mathcal{H}}$
on $M\otimes V^{+}\otimes V^{-}$.

Since we also compute that
\begin{flalign*}
\widetilde{b}^{+}_{n}e^{\mathcal{H}}
\bigl(v_{a}\otimes e(a)\bigr)
&=
\begin{cases}
0\ \  \text{ if }\ n>0 \\
\kappa je^{\mathcal{H}}
\bigl(v_{a}\otimes e(a)\bigr)
\ \ \text{ if }\ n=0,
\end{cases}
\\
\widetilde{b}^{-}_{n}e^{\mathcal{H}}
\bigl(v_{a}\otimes e(a)\bigr)
&=
\begin{cases}
0\ \ \text{ if }\ n>0 \\
\kappa je^{\mathcal{H}}
\bigl(v_{a}\otimes e(a)\bigr)
\ \ \text{ if }\ n=0
\end{cases}
\end{flalign*}
for $n\in\mathbb{Z}_{\geq0}$,
the mapping (\ref{4.2.1}) uniquely extends to
the $V_{k}(\mf{sl}_{2})\otimes \mathcal{F}^{+}\otimes\mathcal{F}^{-}$
-module homomorphism
$\mathscr{F}_{M}\colon
	M\otimes\mathcal{F}_{(0,0)}^{+}
	\otimes\mathcal{F}_{(0,0)}^{-}
	\rightarrow
	\iota_{\mf{sl}_{2}}^{*}(M
	\otimes V^{+}\otimes V^{-}).$
Finally, the injectivity of $\mathscr{F}_{M}$
follows from that of (\ref{4.2.1}).

(2) Denote by $\widetilde{L}_{n},\widetilde{J}_{n},\widetilde{G}^{\pm}_{r}$
 and $\widetilde{b}^{-}_{n}$
 the linear operators
$\iota_{\mf{ns}_{2}}(\mathbf{L})_{(n)},$ $\iota_{\mf{ns}_{2}}(\mathbf{J})_{(n)},$
$\iota_{\mf{ns}_{2}}(\mathbf{G}^{\pm})_{(r+\frac{1}{2})}$
and $\iota_{\mf{ns}_{2}}({\boldsymbol \alpha}^{-})_{(n)}$
on $\iota_{\mf{ns}_{2}}^{*}(N\otimes V^{+}\otimes V^{-})$.
Similarly to (1),
by computations, we have
\begin{flalign*}
\widetilde{L}_{n}e^{\mathcal{J}}
\bigl(w_{a}\otimes e(a)\bigr)
&=
e^{\mathcal{J}}\bigl(
L_{n}w_{a}\otimes e(a)\bigr),
\\
\widetilde{J}_{n}e^{\mathcal{J}}
\bigl(w_{a}\otimes e(a)\bigr)
&=
e^{\mathcal{J}}\bigl(
J_{n}w_{a}\otimes e(a)\bigr),
\\
\widetilde{G}^{+}_{r}e^{\mathcal{J}}
\bigl(w_{a}\otimes e(a)\bigr)
&=
e^{\mathcal{J}}\bigl(
G^{+}_{r}w_{a}\otimes e(a+1)\bigr),
\\
\widetilde{G}^{-}_{r}e^{\mathcal{J}}
\bigl(w_{a}\otimes e(a)\bigr)
&=
e^{\mathcal{J}}\bigl(
G^{-}_{r}w_{a}\otimes e(a-1)
\bigr)
\end{flalign*}
for $n\in\mathbb{Z}$.
We also have
\begin{flalign*}
\widetilde{b}^{+}_{n}e^{\mathcal{J}}
\bigl(w_{a}\otimes e(a)\bigr)
&=
\begin{cases}
0\ \  \text{ if }\ n>0 \\
\kappa je^{\mathcal{J}}
\bigl(w_{a}\otimes e(a)\bigr)
\ \ \text{ if }\ n=0,
\end{cases}
\\
\widetilde{b}^{-}_{n}e^{\mathcal{J}}
\bigl(w_{a}\otimes e(a)\bigr)
&=
\begin{cases}
0\ \ \text{ if }\ n>0 \\
\kappa je^{\mathcal{J}}
\bigl(w_{a}\otimes e(a)\bigr)
\ \ \text{ if }\ n=0
\end{cases}
\end{flalign*}
for $n\in\mathbb{Z}_{\geq0}$.
Therefore the same discussion as in (1) works.
\end{proof}

In fact, the previous embeddings $\mathscr{F}_{M}$ and $\mathscr{G}_{N}$ 
extend to the following isomorphisms:

\begin{prp}\label{key2}
(1) As a $V_{k}(\mf{sl}_{2})
\otimes \mathcal{F}^{+}\otimes\mathcal{F}^{-}$-module,
\begin{equation*}
	\bigoplus_{n,m\in\mathbb{Z}}
	M^{\theta+n+m}
	\otimes\mathcal{F}_{(\theta,n)}^{+}
	\otimes\mathcal{F}_{(\theta+n,m)}^{-}
	\overset{\simeq}{\longrightarrow}
	\iota_{\mf{sl}_{2}}^{*}(M^{\theta}
	\otimes V^{+}\otimes V^{-}).
\end{equation*}

(2) As a $V_{c_{k}}(\mf{ns}_{2})
\otimes \mathcal{F}^{+}\otimes\mathcal{F}^{-}$-module,
\begin{equation*}
	\bigoplus_{n,m\in\mathbb{Z}}N^{\theta+n+m}
	\otimes\mathcal{F}_{(\theta+m,n)}^{+}
	\otimes\mathcal{F}_{(\theta,m)}^{-}
	\overset{\simeq}{\longrightarrow}
	\iota_{\mf{ns}_{2}}^{*}(N^{\theta}
	\otimes V^{+}\otimes V^{-}).
\end{equation*}
\end{prp}

\begin{proof}
(1) Put 
$h(\theta,n,m):=\frac{\theta+n+m}{2}\mathbf{H}
+(\frac{k}{2}\theta-n)\kappa{\boldsymbol \alpha}^{+}
-(\frac{k}{k+2}n+m)\kappa^{-1}{\boldsymbol \alpha}^{-}.$
Then we have
$\iota_{\mf{sl}_{2}}\bigl(h(\theta,n,m)\bigr)
=\frac{\theta}{2}\mathbf{H}+n\alpha^{+}_{-1}+m\alpha^{-}_{-1}$
and the composition
$\iota_{\mf{sl}_{2}}^{*}(\id_{M}\otimes\xi^{+}_{n}\otimes\xi^{-}_{m})
\circ\Delta^{h(\theta,n,m)}(\mathscr{F}_{M})$
gives an embedding
\begin{equation*}
	M^{\theta+n+m}
	\otimes\mathcal{F}_{(\theta,n)}^{+}
	\otimes\mathcal{F}_{(\theta+n,m)}^{-}
	\rightarrow
	\iota_{\mf{sl}_{2}}^{*}(M^{\theta}
	\otimes V^{+}\otimes V^{-}),
\end{equation*}
where $\xi^{+}_{n}$ and $\xi^{-}_{m}$ are defined in Section \ref{HLtw}.
Since each images for $n,m\in\mathbb{Z}$
are in distinct eigenspaces with respect to 
${\boldsymbol \alpha}^{\pm}_{(0)}$,
the sum of these embeddings induces an injective homomorphism
\begin{equation}\label{TWaff}
	\bigoplus_{n,m\in\mathbb{Z}}
	M^{\theta+n+m}
	\otimes\mathcal{F}_{(\theta,n)}^{+}
	\otimes\mathcal{F}_{(\theta+n,m)}^{-}
	\rightarrow
	\iota_{\mf{sl}_{2}}^{*}(M^{\theta}
	\otimes V^{+}\otimes V^{-}).
\end{equation}
By Lemma \ref{character}, 
this injection gives an isomorphism.

(2) Put 
$h'(\theta,n,m):=(\theta+n+m)\mathbf{J}
+(\frac{k}{2}(\theta+m)-n)\kappa{\boldsymbol \alpha}^{+}
-(\frac{k}{k+2}\theta+m)\kappa^{-1}{\boldsymbol \alpha}^{-}.$
Then we have
$\iota_{\mf{ns}_{2}}\bigl(h'(\theta,n,m)\bigr)
=\theta\mathbf{J}+n\alpha^{+}_{-1}+m\alpha^{-}_{-1}$
and the composition
$\iota_{\mf{ns}_{2}}^{*}(\id_{N}\otimes\xi^{+}_{n}\otimes\xi^{-}_{m})
\circ\Delta^{h'(\theta,n,m)}(\mathscr{G}_{N})$
gives rise to an embedding
\begin{equation}\label{TWsca}
	\bigoplus_{n,m\in\mathbb{Z}}N^{\theta+n+m}
	\otimes\mathcal{F}_{(\theta+m,n)}^{+}
	\otimes\mathcal{F}_{(\theta,m)}^{-} \rightarrow
	\iota_{\mf{ns}_{2}}^{*}(N^{\theta}
	\otimes V^{+}\otimes V^{-})
\end{equation}
in the same way as (1).
By Lemma \ref{character}, this gives an isomorphism.
\end{proof}

\subsection{Construction of natural transformation}

By the following proposition, we complete the proof of Theorem \ref{main}.

\begin{prp}
(1) The assignments
\begin{equation*}
\mathscr{F}(M)\colon
M\rightarrow
\Omega^{-}_{j}\circ
\Omega^{+}_{j}(M);\,
v\mapsto
e^{\mathcal{H}}
\left(
\sum_{a\in\mathbb{Z}}v_{a}\otimes e^{a\alpha^{+}}
\otimes e^{-a\alpha^{-}}\right),
\end{equation*}
\begin{equation*}
\mathscr{G}(N)\colon
N\rightarrow
\Omega^{+}_{j}\circ\Omega^{-}_{j}(N);\,
w\mapsto
e^{\mathcal{J}}
\left(
\sum_{a\in\mathbb{Z}}w_{a}\otimes e^{a\alpha^{+}}
\otimes e^{-a\alpha^{-}}\right)
\end{equation*}
give a $V_{k}(\mf{sl}_{2})$-module isomorphism
and a $V_{c_{k}}(\mf{ns}_{2})$-module isomorphism, respectively.

(2) The sets of isomorphisms $\mathscr{F}$ and $\mathscr{G}$
give natural transformations
${\sf id}_{\mathcal{C}_{V_{k}(\mf{sl}_{2})}^{\overline{(h,j)}}}
\simeq
\Omega^{-}_{j}\circ
\Omega^{+}_{j}$
and
${\sf id}_{\mathcal{C}_{V_{c_{k}}(\mf{ns}_{2})}
^{\overline{(h-\frac{j^{2}}{k+2},\frac{2j}{k+2})}}}
\simeq\Omega^{+}_{j}\circ\Omega^{-}_{j}$,
respectively.
\end{prp}

\begin{proof}
(1) Both statements follow from
Lemma \ref{twvacc}, \ref{key} and Proposition \ref{key2}.

(2) Let $f\in\Hom_{V_{k}(\mf{sl}_{2})}(M,M')$
and $g\in\Hom_{V_{c_{k}}(\mf{ns}_{2})}(N,N')$.
It suffices to show that the following diagrams
\[\xymatrix@!C=54pt{
M \ar[d]_{\mathscr{F}(M)} 
 \ar[r]^{f}
 \ar@{}[dr] &
M' \ar[d]^{\mathscr{F}(M')} & 
N \ar[d]_{\mathscr{G}(N)} 
 \ar[r]^{g}
 \ar@{}[dr] &
N' \ar[d]^{\mathscr{G}(N')}
\\
\Omega_{j}^{+}(M)
 \ar[r]^{\Omega_{j}^{+}(f)} &
\Omega_{j}^{+}(M') &
\Omega_{j}^{-}(N)
 \ar[r]^{\Omega_{j}^{-}(g)} &
\Omega_{j}^{-}(N')
\\
}\]
commute.
Since $f$ commutes with the action of $V_{k}(\mf{sl}_{2})$,
we have
$\bigl(f(v)\bigr)_{a}=f(v_{a})$ for $v\in M$ and $a\in\mathbb{Z}$.
It also follows that
the operators $\Omega_{j}^{+}(f)$ and $e^{\mathcal{H}}$
 commute in $\End(M\otimes V^{+}\otimes V^{-})$.
Then the former diagram commutes.
The latter is proved in the same way.
\end{proof}

\section{Spectral flow equivariance}

In this section we discuss the relationship 
between the functors $\Omega^{\pm}_{j}$ and the spectral flow automorphisms.
To simplify notation, we put
\begin{flalign*}
&h(\theta,r):=(\theta+r)\mathbf{J}+\left(\frac{k}{2}\theta-r\right)
\kappa{\boldsymbol \alpha}^{+}
\in V_{c_{k}}(\mf{ns}_{2})\otimes \mathcal{F}^{+}, \\
&h'(\theta,r):=
\frac{\theta+r}{2}\mathbf{H}-\left(\frac{k}{k+2}\theta+r\right)
\kappa^{-1}{\boldsymbol \alpha}^{-}
\in V_{k}(\mf{sl}_{2})\otimes\mathcal{F}^{-}
\end{flalign*}
for $\theta,r\in\mathbb{Z}$.

\begin{prp}\label{5.1}
(1) As a $V_{c_{k}}(\mf{ns}_{2})\otimes \mathcal{F}^{+}$-module,
\begin{equation}\label{br11}
\bigoplus_{n\in\mathbb{Z}}
\Omega^{+}_{j}(M)^{n}
\otimes\mathcal{F}_{(0,n)}^{+}
\simeq
\iota_{+}^{*}(M\otimes V^{+}).
\end{equation}

(2) As a $V_{k}(\mf{sl}_{2})\otimes\mathcal{F}^{-}$-module,
\begin{equation}\label{br22}
\bigoplus_{m\in\mathbb{Z}}
\Omega^{-}_{j}(N)^{m}
\otimes\mathcal{F}_{(0,m)}^{-}
\simeq
\iota_{-}^{*}(N\otimes V^{-}).
\end{equation}
\end{prp}

\begin{proof}
(1) By using Lemma \ref{vacc} (1), it suffices to show that
$\Omega^{+}_{j}(M)^{n}\simeq\Omega^{+}_{j-n}(M)$
as a $V_{c_{k}}(\mf{ns}_{2})$-module.
Let $r\in\mathbb{Z}$ and denote by
$f_{M}$ the isomorphism in (\ref{vac1}).
Since 
$\iota_{+}\bigl(h(0,r)\bigr)
=r\alpha^{+}_{-1}\in V_{k}(\mf{sl}_{2})\otimes V^{+}$,
we get the $V_{c_{k}}(\mf{ns}_{2})\otimes \mathcal{F}^{+}$-module
 isomorphism
\begin{equation*}
\Delta^{h(0,r)}(f_{M})\colon\bigoplus_{n\in\mathbb{Z}}
\Omega^{+}_{j-n}(M)^{r}\otimes \mathcal{F}^{+}_{(0,n+r)}
\overset{\simeq}{\longrightarrow}
\iota_{+}^{*}\bigl(M\otimes (V^{+})^{r}\bigr).
\end{equation*}
By using the $V^{+}$-module isomorphism
$\xi^{+}_{r}\colon(V^{+})^{r}\overset{\simeq}
{\longrightarrow} V^{+}$,
the composition 
$(f_{M})^{-1}
\circ\iota^{*}_{+}(\id_{M}\otimes\xi^{+}_{r})
\circ\Delta^{h(0,r)}(f_{M})$
gives a $V_{c_{k}}(\mf{ns}_{2})\otimes \mathcal{F}^{+}$-module
 isomorphism
\begin{equation*}
\bigoplus_{n\in\mathbb{Z}}
\Omega^{+}_{j-n}(M)^{r}\otimes \mathcal{F}^{+}_{(0,n+r)}
\overset{\simeq}{\longrightarrow}
\bigoplus_{n\in\mathbb{Z}}
\Omega^{+}_{j-n-r}(M)\otimes \mathcal{F}^{+}_{(0,n+r)}.
\end{equation*}
The restriction of this mapping to the 
invariant subspace
$\Omega^{+}_{j}(M)^{r}\otimes\mathbb{C}\ket{(0,r)}^{+}$
gives a $V_{c_{k}}(\mf{ns}_{2})$-module isomorphism
\begin{equation*}
\Omega^{+}_{j}(M)^{r}\otimes\mathbb{C}\ket{(0,r)}^{+}
\overset{\simeq}{\longrightarrow}
\Omega^{+}_{j-r}(M)\otimes\mathbb{C}\ket{(0,r)}^{+}.
\end{equation*}
This completes the proof.

(2) Applying the functor $\Delta^{h'(0,r)}$
to the isomorphism (\ref{vac2}),
we also get the isomorphism
\begin{equation*}
\bigoplus_{m\in\mathbb{Z}}
\Omega^{-}_{j+\frac{k+2}{2}m}(N)^{r}\otimes\mathcal{F}^{-}_{(0,m+r)}
\overset{\simeq}{\longrightarrow}
\iota_{-}^{*}\bigl(N\otimes (V^{-})^{r}\bigr).
\end{equation*}
Since $(V^{-})^{r}\simeq V^{-}$ as a $V^{-}$-module, 
we obtain
$$
\Omega^{-}_{j}(N)^{r}\otimes\mathbb{C}\ket{(0,r)}^{-}
\overset{\simeq}{\longrightarrow}
\Omega^{-}_{j+\frac{k+2}{2}r}(N)\otimes\mathbb{C}\ket{(0,r)}^{-}
$$
 in the same way as (1).
\end{proof}

\begin{cor}
(1) As a $V_{c_{k}}(\mf{ns}_{2})\otimes \mathcal{F}^{+}$-module,
\begin{equation*}
\bigoplus_{n\in\mathbb{Z}}\Omega^{+}_{j}(M)^{\theta+n}
\otimes\mathcal{F}_{(\theta,n)}^{+}
\simeq
\iota_{+}^{*}(M^{\theta}\otimes V^{+}).
\end{equation*}

(2) As a $V_{k}(\mf{sl}_{2})\otimes\mathcal{F}^{-}$-module,
\begin{equation*}
\bigoplus_{m\in\mathbb{Z}}\Omega^{-}_{j}(N)^{\theta+m}
\otimes\mathcal{F}_{(\theta,m)}^{+}
\simeq
\iota_{-}^{*}(N^{\theta}\otimes V^{+}).
\end{equation*}
\end{cor}

\begin{proof}
The statement follows by applying the functors $\Delta^{h(\theta,0)}$ and $\Delta^{h'(\theta,0)}$
to the isomorphisms (\ref{br11}) and (\ref{br22}), respectively.
\end{proof}

As a consequence, we obtain the following ``spectral flow equivariance'' property
of the functors $\Omega^{+}_{j}$ and $\Omega^{-}_{j}$.

\begin{cor}\label{sfeq}
For $a,b\in\mathbb{Z}$, the following diagrams
\[\xymatrix@!C=84pt{
\mathcal{C}^{\overline{(h,j;0)}}_{V_{k}(\mf{sl}_{2})} \ar[d]_{\Delta^{a}_{\mf{sl}_{2}}} 
 \ar[r]^{\Omega^{+}_{j}}
 &
\mathcal{C}^{\overline{(h,j;0)}}_{V_{c_{k}}(\mf{ns}_{2})} \ar[d]^{\Delta^{a+b}_{\mf{ns}_{2}}} \\
\mathcal{C}^{\overline{(h,j;a)}}_{V_{k}(\mf{sl}_{2})}
 \ar[r]^{\Omega_{j+\frac{k}{2}a-b}^{+}} &
\mathcal{C}^{\overline{(h,j;a+b)}}_{V_{c_{k}}(\mf{ns}_{2})} \\
}\]
\[\xymatrix@!C=84pt{
\mathcal{C}^{\overline{(h,j;0)}}_{V_{c_{k}}(\mf{ns}_{2})} \ar[d]_{\Delta^{a}_{\mf{ns}_{2}}} 
 \ar[r]^{\Omega^{-}_{j}}
 &
\mathcal{C}^{\overline{(h,j;0)}}_{V_{k}(\mf{sl}_{2})} \ar[d]^{\Delta^{a+b}_{\mf{sl}_{2}}} \\
\mathcal{C}^{\overline{(h,j;a)}}_{V_{c_{k}}(\mf{ns}_{2})}
 \ar[r]^-{\Omega^{-}_{j+\frac{k}{2}a+\frac{k+2}{2}b}} &
\ \mathcal{C}^{\overline{(h,j;a+b)}}_{V_{k}(\mf{sl}_{2})} \\
}\]
commute up to natural equivalence.
\end{cor}

\begin{ex}
Setting $a=0$ in the former diagram, we obtain
 $\Omega_{j-b}^{+}\simeq
\Delta^{b}_{\mf{ns}_{2}}\circ\Omega^{+}_{j}$
as a functor from $\mathcal{C}^{\overline{(h,j)}}_{V_{k}(\mf{sl}_{2})}$
to $\mathcal{C}_{V_{c_{k}}(\mf{ns}_{2})}^{\overline{(h,j;b)}}$.
\end{ex}

\section{Applications}

\subsection{Examples of corresponding pairs}

In this subsection we show that relaxed highest weight 
$\widehat{\mf{sl}}_{2}$-modules of level $k$
exactly correspond to highest weight $\mf{ns}_{2}$-modules of
central charge $c_{k}$.
To simplify notation, we write
$\mathcal{M}_{(h,j,k)}$ and $\mathcal{L}_{(h,j,k)}$ for 
$\mathcal{M}_{h-\frac{j^{2}}{k+2},\frac{2j}{k+2},c_{k}}$
and $\mathcal{L}_{h-\frac{j^{2}}{k+2},\frac{2j}{k+2},c_{k}}$,
respectively.

\begin{prp}\label{cpair}
We have the following isomorphisms:

(1) $\Omega^{+}_{j}(R_{h,j,k})\simeq
\mathcal{M}_{(h,j,k)}$,

(2) $\Omega^{+}_{j}(M_{j,k})\simeq\mathcal{M}^{+}_{\frac{2j}{k+2},c_{k}}$,

(3) $\Omega^{+}_{0}\bigl(V_{k}(\mf{sl}_{2})\bigr)\simeq V_{c_{k}}(\mf{ns}_{2})$,

(4) $\Omega^{+}_{j}(L_{h,j,k})\simeq
\mathcal{L}_{(h,j,k)}$.
\end{prp}

\begin{proof}
Let $\theta\in\mathbb{Z}$. Denote by
$\ket{h,j,k;\,\theta}$ and
$\ket{(h,j,k);\,\theta}^{\mf{ns}_{2}}$
the canonical generators of
$R_{h,j,k}^{\theta}$
and $\mathcal{M}^{\theta}_{(h,j,k)}$, respectively.
By the annihilation relations,
the $V_{c_{k}}(\mf{ns}_{2})\otimes \mathcal{F}^{+}$-module
homomorphism 
\begin{equation*}
f(\theta)\colon\bigoplus_{n\in\mathbb{Z}}
\mathcal{M}^{\theta+n}_{(h,j,k)}
\otimes\mathcal{F}_{(\theta,n)}^{+}
\rightarrow
\iota_{+}^{*}(R_{h,j,k}^{\theta}\otimes V^{+})
\end{equation*}
is uniquely determined by sending
$$\textstyle
\ket{(h,j,k);\,\theta+n}^{\mf{ns}_{2}}
\otimes\ket{(\theta,n)}^{+}
\mapsto
\ket{h,j,k;\,\theta}\otimes e^{n\alpha^{+}}$$
for $n\in\mathbb{Z}$.
In the same way, the $V_{k}(\mf{sl}_{2})\otimes\mathcal{F}^{-}$-module
homomorphism
\begin{equation*}
g(\theta)\colon 
\bigoplus_{m\in\mathbb{Z}}
R_{h,j,k}^{\theta+m}
\otimes\mathcal{F}_{(\theta,m)}^{-}
\rightarrow
\iota_{-}^{*}(\mathcal{M}_{(h,j,k)}^{\theta}
\otimes V^{-}).
\end{equation*}
is uniquely determined by sending
$$\textstyle
\ket{h,j,k;\,\theta+m}
\otimes\ket{(\theta,m)}^{-}
\mapsto
\ket{(h,j,k);\,\theta}^{\mf{ns}_{2}}\otimes e^{m\alpha^{-}}$$
for $m\in\mathbb{Z}$.
Since we have
$$e^{\mathcal{H}}\bigl(\ket{h,j,k;\theta}
\otimes e^{n\alpha^{+}}\otimes e^{m\alpha^{-}}\bigr)
=\ket{h,j,k;\theta}\otimes e^{n\alpha^{+}}\otimes e^{m\alpha^{-}}$$
for any $n,m\in\mathbb{Z}$,
the $V_{k}(\mf{sl}_{2})
\otimes \mathcal{F}^{+}\otimes\mathcal{F}^{-}$-module homomorphism
\begin{flalign*}
&\bigoplus_{n,m\in\mathbb{Z}}
R_{h,j,k}^{\theta+n+m}
\otimes\mathcal{F}_{(\theta,n)}^{+}
\otimes\mathcal{F}_{(\theta+n,m)}^{-}\\
\xrightarrow{\bigoplus g(\theta+n)}\ \ 
&(\iota_{-})_{13}^{*}\left(
\bigoplus_{n\in\mathbb{Z}}
\mathcal{M}^{\theta+n}_{(h,j,k)}
\otimes\mathcal{F}_{(\theta,n)}^{+}\otimes V^{-}
\right) \\
\xrightarrow{\hspace{4mm}f(\theta)\hspace{3.5mm}}\ \ 
&(\iota_{-})_{13}^{*}\circ(\iota_{+})_{12}^{*}\bigl(
R_{h,j,k}^{\theta}
\otimes V^{+}\otimes V^{-}\bigr)
\end{flalign*}
coincides with the $V_{k}(\mf{sl}_{2})
\otimes \mathcal{F}^{+}\otimes\mathcal{F}^{-}$-module
isomorphism (\ref{TWaff}).
Therefore $g(\theta)$ is injective and $f(\theta)$ is surjective
for any $\theta\in\mathbb{Z}$.
Similarly, since
$$e^{\mathcal{J}}\bigl(\ket{(h,j,k);\theta}^{\mf{ns}_{2}}
\otimes e^{n\alpha^{+}}\otimes e^{m\alpha^{-}}\bigr)
=\ket{(h,j,k);\theta}^{\mf{ns}_{2}}\otimes e^{n\alpha^{+}}\otimes e^{m\alpha^{-}},$$
the $V_{c_{k}}(\mf{ns}_{2})
\otimes \mathcal{F}^{+}\otimes\mathcal{F}^{-}$-module homomorphism
\begin{flalign*}
&\bigoplus_{n,m\in\mathbb{Z}}
\mathcal{M}_{(h,j,k)}^{\theta+n+m}
\otimes\mathcal{F}_{(\theta+m,n)}^{+}
\otimes\mathcal{F}_{(\theta,m)}^{-}\\
\xrightarrow{\bigoplus f(\theta+m)}\ \ 
&(\iota_{+})_{12}^{*}\left(
\bigoplus_{m\in\mathbb{Z}}
R^{\theta+m}_{(h,j,k)}
\otimes V^{+}\otimes\mathcal{F}_{(\theta,m)}^{-}
\right) \\
\xrightarrow{\hspace{4mm}g(\theta)\hspace{4.5mm}}\ \ 
&(\iota_{+})_{12}^{*}\circ(\iota_{-})_{13}^{*}\bigl(
\mathcal{M}_{(h,j,k)}^{\theta}
\otimes V^{+}\otimes V^{-}\bigr)
\end{flalign*}
coincides with the $V_{c_{k}}(\mf{ns}_{2})
\otimes \mathcal{F}^{+}\otimes\mathcal{F}^{-}$-module
isomorphism (\ref{TWsca}).
Therefore $g(\theta)$ is surjective and $f(\theta)$ is injective.
This completes the proof of (1).
Since (2) and (3) are proved in the same way,
we omit the proof.

Finally, as the functor $\Omega^{\pm}_{j}$ gives a categorical equivalence,
the module $\Omega^{+}_{j}(L_{h,j,k})$
is isomorphic to the simple quotient of 
$\Omega^{+}_{j}(R_{h,j,k})\simeq\mathcal{M}_{(h,j,k)}$.
\end{proof}

\begin{rem}
It is known that the VOA $V_{k}(\mf{sl}_{2})$
has a non-trivial ideal if and only if $k$ is an admissible level
for $\widehat{\mf{sl}}_{2}$, {\it i.e.} there exists 
a pair of coprime integers $(p,p')\in\mathbb{Z}_{\geq2}\times\mathbb{Z}_{\geq1}$
such that $k=-2+\frac{p}{\ p'}$ (see \cite[Theorem 0.2.1]{gorelik2007simplicity}).
It follows from Proposition \ref{cpair} that
the VOSA $V_{c}(\mf{ns}_{2})$ has a non-trivial ideal
if and only if there exists an admissble level $k$ for $\widehat{\mf{sl}}_{2}$
such that $c=c_{k}=3(1-\frac{2p'}{p})$.
This is already proved by
M. Gorelik and V. Kac in \cite[Corollary 9.1.5 (ii)]{gorelik2007simplicity}
and we gave another proof.
\end{rem}

By definition of the functors $\Omega^{\pm}_{0}$,
we determine the full commutants as follows:

\begin{cor}\label{CST}
For any $k\in\mathbb{C}\setminus\{-2\}$, we have
\begin{enumerate}
\item ${\sf Com}\bigl(\iota_{+}( \mathcal{F}^{+}),V_{\mf{sl}_{2}}\otimes V^{+}\bigr)\simeq
V_{\mf{ns}_{2}},$
\item ${\sf Com}\bigl(\iota_{-}(\mathcal{F}^{-}),
V_{\mf{ns}_{2}}\otimes V^{-}\bigr)\simeq V_{\mf{sl}_{2}}$.
\end{enumerate}
\end{cor}

\begin{rem}
Since the simplicity of 
${\sf Com}\bigl(\iota_{+}( \mathcal{F}^{+}),L_{k}(\mf{sl}_{2})\otimes V^{+}\bigr)$
and ${\sf Com}\bigl(\iota_{-}(\mathcal{F}^{-}),$
$L_{c_{k}}(\mf{ns}_{2})\otimes V^{-}\bigr)$ follows from a super analog of
\cite[Proposition 3.2]{creutzig2016schur} (see also \cite[Lemma 2.1]{arakawa2017orbifolds}),
the former isomorphism is already obtained by T.\,Creutzig and A.\,Linshaw as a corollary of \cite[Lemma 8.6]{CREUTZIG2019396}.
In particular, the assumption on $k$ in \cite[Lemma 8.7]{CREUTZIG2019396} is not needed.
\end{rem}

\begin{cor}
Assume that $k\neq0$.
Then $\{\mathcal{L}_{h,j,c_{k}}\,|\, h,j\in\mathbb{C}\,\}$
provides the complete representatives of
the equivalence classes of simple objects in $\mathcal{C}_{V_{c_{k}}(\mf{ns}_{2})}$.
\end{cor}

\begin{proof}
Let $L$ be a simple object 
in $\mathcal{C}_{V_{c_{k}}(\mf{ns}_{2})}$.
By Fact \ref{simple}, Corollary \ref{sfeq} and Proposition \ref{cpair} (4),
there exist $h',j'\in\mathbb{C}$ and $\theta\in\mathbb{Z}$ such that
$L\simeq\mathcal{L}_{h',j',c_{k}}^{\theta}$.
Since the set of eigenvalues with respect to $L_{0}$ on 
$\mathcal{L}_{h',j',c_{k}}^{\theta}$ is lower bounded,
there exist unique $h,j\in\mathbb{C}$ such that
$\mathcal{L}_{h',j',c_{k}}^{\theta}\simeq\mathcal{L}_{h,j,c_{k}}$.
\end{proof}

\subsection{Application 1:\,Equivalence in the simple case}

As a corollary of Proposition \ref{cpair} (4), 
we get another proof of the following theorem
obtained by \cite[Theorem 4.1 and 5.1]{Ad99}.

\begin{theom}
The homomorphisms $\iota_{+}$ and $\iota_{-}$ in Theorem
\ref{coset1}
factor through the corresponding simple quotients.
\end{theom}

\begin{proof}
By Proposition \ref{cpair} (4), 
$\Omega^{+}_{0}\bigl(L_{k}(\mf{sl}_{2})\bigr)$
and $\Omega^{-}_{0}\bigl(L_{c_{k}}(\mf{ns}_{2})\bigr)$ are isomorphic
to $L_{c_{k}}(\mf{ns}_{2})$ and $L_{k}(\mf{sl}_{2})$, respectively.
Then the restrictions of the isomorphisms (\ref{vac1}) and (\ref{vac2})
to the $0$-th components are identified with 
the above mappings $\iota_{+}$ and $\iota_{-}$, respectively.
\end{proof}

In a similar way, we get the simple quotient version of Theorem \ref{main}.

\begin{theom}\label{mainsimple}
For any $k\in\mathbb{C}\setminus\{0,-2\}$,
the restrictions of
$\Omega^{+}_{j}$ and $\Omega^{-}_{j}$
in Theorem \ref{main}
also give the categorical equivalences
$$\mathcal{C}_{L_{k}(\mf{sl}_{2})}^{\overline{(h,j)}}
\overset{\simeq}{\longrightarrow}\mathcal{C}_{L_{c_{k}}(\mf{ns}_{2})}
^{\overline{(h-\frac{j^{2}}{k+2},\frac{2j}{k+2})}}\ \text{ and }\ 
\mathcal{C}_{L_{c_{k}}(\mf{ns}_{2})}^{\overline{(h-\frac{j^{2}}{k+2},\frac{2j}{k+2})}}
\overset{\simeq}{\longrightarrow}
\mathcal{C}_{L_{k}(\mf{sl}_{2})}^{\overline{(h,j)}},$$
respectively.
\end{theom}

\subsection{Application 2:\,Character formulas}\label{characterformula}

\subsubsection{At general levels}

In this subsection we fix $h,j\in\mathbb{C}$.
For a $V_{k}(\mf{sl}_{2})$-module $M$ 
and a $V_{c_{k}}(\mf{ns}_{2})$-module $N$ satisfying (\ref{assump}),
we denote by $M_{n}(q)$ and $N_{n}(q)$ 
the elements of $\mathbb{C}(\!(q)\!)$ such that
$$\ch(M)(q,x)=q^{h}x^{j}\sum_{n\in\mathbb{Z}}M_{n}(q)x^{n},\ 
\ch(N)(q,x)=q^{h-\frac{j^{2}}{k+2}}x^{\frac{2j}{k+2}}
\sum_{n\in\mathbb{Z}}N_{n}(q)x^{n}.$$

\begin{theom}\label{cForm}
We have the following:

(1) $\displaystyle
\ch\left(\Omega^{+}_{j}(M)\right)(q,x)=
q^{h-\frac{j^{2}}{k+2}}x^{\frac{2j}{k+2}}
\sum_{n\in\mathbb{Z}}q^{\frac{\ n^2}{2}}M_{n}(q)x^{n}.$

(2) $\displaystyle
\ch\left(\Omega^{-}_{j}(N)\right)(q,x)=
q^{h}x^{j}
\sum_{n\in\mathbb{Z}}q^{-\frac{\ n^2}{2}}N_{n}(q)x^{n}$.
\end{theom}

\begin{proof}
(1) By Proposition \ref{5.1}.1 and easy calculation, we have
\begin{equation}\label{7.1}
\begin{array}{cl}
\ch\left(\iota_{+}^{*}(M\otimes V^{+})\right)(q,x_{1},x_{2})
&\displaystyle=\sum_{n\in\mathbb{Z}}
\ch\left(\Omega_{j}^{+}(M)^{n}\right)(q,x_{1})
\ch( \mathcal{F}^{+}_{\kappa(j-n)})(q,x_{2})
\\
&=\ch(M)(q,x_{1}^{\kappa^{2}}x_{2}^{\kappa})
\ch(V^{+})(q,x_{1}^{1-\kappa^{2}}x_{2}^{-\kappa}).
\end{array}
\end{equation}
By the character formula for fermionic Fock module,
we have
$$\ch(V^{+})(q,x_{1}^{1-\kappa^{2}}x_{2}^{-\kappa})
=\sum_{m\in\mathbb{Z}}q^{\frac{\ m^{2}}{2}}
x_{1}^{(1-\kappa^{2})m}x_{2}^{-\kappa m}\times(q;q)^{-1}_{\infty}.$$
Multiplying $q^{-\frac{j^{2}}{k+2}}(q;q)_{\infty}$ to (\ref{7.1})
and comparing the coefficients of $x_{2}^{\kappa j}$,
we obtain the required formula.

(2) The equality follows from (1) and $\Omega^{+}_{j}\circ\Omega^{-}_{j}(N)\simeq N$.
\end{proof}

\begin{rem}
The character formula for irreducible highest weight $\widehat{\mf{sl}}_{2}$-modules
at non-critical level ({\it i.e.} $k\neq-2$) is deduced by \cite{malikov1991verma}.
Therefore we obtain all the corresponding character formula for the simple
quotients of chiral Verma modules.
\end{rem}

\subsubsection{At admissible levels}

We fix a pair of coprime integers 
$(p,p')\in\mathbb{Z}_{\geq2}\times\mathbb{Z}_{\geq1}$
and a pair of integers $(r,s)$ such that $1\leq r\leq p-1$ 
and $0\leq s\leq p'-1$.
Throughout this subsection, we always assume
that $k=-2+\frac{p}{\ p'}$.
It follows that $c_{k}=3\left(1-\frac{2p'}{p}\right)$.

We use the following notations:
$$M(n,m):=M_{\frac{n-1}{2}-\frac{p}{\ p'}\frac{m}{2},k},\ 
\mathcal{M}^{+}(n,m):=\mathcal{M}^{+}_{\frac{\ p'}{p}(n-1)-m,c_{k}}$$
for $n,m\in\mathbb{Z}$.
Denote by $L(n,m)$ and $\mathcal{L}(n,m)$ the simple quotients
of $M(n,m)$ and $\mathcal{M}^{+}(n,m)$, respectively.
Then the following BGG resolution 
of $L(r,s)$ is obtained by 
F. Malikov.

\begin{fct}\cite[Theorem A]{malikov1991verma}
We set $M_{n}:=M(n)\oplus M(-n)$ for $n>0$,
where
$M(2m):=M(2pm+r,s)$ and $M(2m-1):=M(2pm-r,s)$
for $m\in\mathbb{Z}$.
Then there exists an exact sequence
$$\cdots
\rightarrow M_{n}
\rightarrow \cdots
\rightarrow M_{2}
\rightarrow M_{1}
\rightarrow M(r,s)
\rightarrow L(r,s)
\rightarrow 0.$$
\end{fct}

%\begin{rem}
%The fact also follows from
%the existence of the Bernstein-Gel'fand-Gel'fand (BGG) resolutions for 
%integrable highest weight modules over affine Lie algebras 
%(see Section 8 in \cite{garland1976lie}) and
%Fiebig's equivalence between certain blocks of the BGG category $\mathcal{O}$ 
%\cite[Theorem 11]{fiebig2006combinatorics}.
%For general untwisted affine Lie algebras $\widehat{\mf{g}}$,
%the existence of the BGG-type resolution for
%the principal admissible modules of $\widehat{\mf{g}}$
%follows from \cite{garland1976lie} and 
%\cite[Theorem 11]{fiebig2006combinatorics}.
%\end{rem}

Applying the functor $\Omega^{+}_{\frac{r-1}{2}-\frac{p}{\ p'}\frac{s}{2}}$
to the above resolution,
 we obtain by Corollary \ref{sfeq} another proof of 
the following BGG-type resolution
for the $\mathcal{N}=2$ SCA module $\mathcal{L}(r,s)$
given by \cite[Theorem 3.1]{FSST99}.

\begin{theom}\label{sBGG}
We set $\mathcal{M}^{+}_{n}:=\mathcal{M}^{+}(n)\oplus \mathcal{M}^{+}(-n)$ for $n>0$,
where
$\mathcal{M}^{+}(2m):=\mathcal{M}^{+}(2pm+r,s)^{pm}$
and $\mathcal{M}^{+}(2m-1):=\mathcal{M}^{+}(2pm-r,s)^{pm-r}$
for $m\in\mathbb{Z}$.
Then there exists an exact sequence
$$\cdots
\rightarrow \mathcal{M}_{n}^{+}
\rightarrow \cdots
\rightarrow \mathcal{M}_{2}^{+}
\rightarrow \mathcal{M}_{1}^{+}
\rightarrow \mathcal{M}^{+}(r,s)
\rightarrow \mathcal{L}(r,s)
\rightarrow 0.$$
\end{theom}

\begin{rem}
(1) The proof in \cite{FSST99} is essentially based on the `equivalence' 
in \cite[Theorem IV.3]{FST98}.
However, an explicit construction of the categorical equivalence is not given
in the literature to the best of our knowledge.

(2) When $p'=1$, every irreducible 
unitarizable highest weight $\mf{ns}_{2}$-module 
of central charge $c=3(1-\frac{2}{p})$ is
of the form $\mathcal{L}(r,0)^{\theta}$
for some $1\leq r\leq p-1$ and $0\leq\theta\leq r-1$.
The highest weight $(h,j,c)$ of $\mathcal{L}(r,0)^{\theta}$ is given by
$$(h,j)=\left(\frac{(r-\theta-\frac{1}{2})(\theta+\frac{1}{2})-\frac{1}{4}}{p},
\frac{(r-\theta-\frac{1}{2})-(\theta+\frac{1}{2})}{p}\right).$$
\end{rem}

We set $\displaystyle \vartheta(q,z)
:=(-zq^{\frac{1}{2}};q)_{\infty}(-z^{-1}q^{\frac{1}{2}};q)_{\infty}(q;q)_{\infty}$
%=\sum_{n\in\mathbb{Z}}z^{n}q^{\frac{\ n^{2}}{2}}
and
$\eta(q):=q^{\frac{1}{24}}(q;q)_{\infty}.$
%$|x|<|q|^{-\theta-\frac{1}{2}}$, {\it i.e.}
%$$|x|<|q|^{-\frac{1}{2}},|q|^{-\frac{3}{2}},|q|^{-\frac{5}{2}}\ldots,$$
% or $|q|^{-\theta-\frac{1}{2}}<|x|$, {\it i.e.}
%$$\ldots,|q|^{\frac{5}{2}},|q|^{\frac{3}{2}},|q|^{\frac{1}{2}}<|x|.$$
Then, as a corollary of Proposition \ref{sBGG}, 
we reprove the following character formula
for $\widetilde{\ch}\bigl(\mathcal{L}(r,s)\bigr)
:=q^{-\frac{c_{k}}{24}}\ch\bigl(\mathcal{L}(r,s)\bigr)$
obtained by \cite[Theorem 4.8]{FSST99}.

\begin{theom}\label{FSSTCF}
Put $j=\frac{r-1}{2}-\frac{p}{\ p'}\frac{s}{2}$. Then we have
$$\widetilde{\ch}\bigl(\mathcal{L}(r,s)\bigr)
=q^{-\frac{\ p'}{p}j^{2}}x^{\frac{2p'}{p}j}
\frac{\vartheta(q,x)}{\eta(q)^{3}}\Phi_{p,p';r,s}(q,x),$$
where $\Phi_{p,p';r,s}(q,x)$ is the expansion of the meromorphic function
$$\sum_{n\in\mathbb{Z}}
\left(\frac{q^{\frac{\ p'}{p}(pn+\frac{1}{2}+j)^{2}}}{1+xq^{pn+\frac{1}{2}}}
-\frac{q^{\frac{\ p'}{p}(pn-r+\frac{1}{2}+j)^{2}}}{1+xq^{pn-r+\frac{1}{2}}}\right)$$
in the region $\mathbb{A}$ (see Remark \ref{divergent} for the definition).
\end{theom}

\begin{proof}
By the Poincar\'{e}-Birkoff-Witt theorem and Lemma \ref{twchar}, we have
$$q^{-\frac{c_{k}}{24}}
\ch\left(\mathcal{M}^{+}(2\theta+r,s)^{\theta}\right)
=q^{-\frac{\ p'}{p}j^{2}}
x^{\frac{2p'}{p}j}\frac{\vartheta(q,x)}{\eta(q)^{3}}
\frac{q^{\frac{\ p'}{p}(\theta+\frac{1}{2}+j)^{2}}}
{1+xq^{\theta+\frac{1}{2}}}$$
for any $\theta\in\mathbb{Z}$.
Thus the formula follows from the BGG-type resolution.
\end{proof}

\begin{rem}
(1) When $p'=1$, the character formula is well known and can be written in terms of theta functions
(see \cite{dobrev1987characters}, \cite{kiritsis1988character}, \cite{matsuo1987character}).
See \cite[\S5]{carpi2015n} for the proof without using the embedding diagram (cf.\,\cite{dorrzapf1998embedding}).

(2) For general $p'\geq1$, the formal series
$$\Phi_{p,p';r,s}(q,x)=
\left(\sum_{n,m\geq0}-\sum_{n,m<0}\right)
(-x)^{m}\varphi_{p,p';r,s}^{n,m}(q),$$
where $$\varphi_{p,p';r,s}^{n,m}(q):=q^{\frac{\ p'}{p}(pn+\frac{1}{2}+j)^{2}+(pn+\frac{1}{2})m}
-q^{\frac{\ p'}{p}(pn+p-r+\frac{1}{2}+j)^{2}+(pn+p-r+\frac{1}{2})m},$$
is absolutely convergent in the region $\mathbb{A}$ (see Remark \ref{divergent}).
Note that this function is related with a certain specialization of 
the mock theta function of type $A(1,0)$ defined in \cite{kac2014representations}.
See \cite[Proposition 9.2]{kac2014representations} for example.
\end{rem}

\appendix

\section{Vertex superalgebras}\label{Cvertex superalgebra}

In this section we recall some facts
about conformal vertex superalgebras and their modules
to fix our notation. See \cite{Kac97V} for the details.

\subsection{Definitions}

Let $V=V^{\bar{0}}\oplus V^{\bar{1}}$
be a $\mathbb{Z}_{2}$-graded $\mathbb{C}$-vector space,
$Y(-;z)\colon V\rightarrow\End(V)[\![z^{\pm1}]\!]$
an even linear mapping
and ${\bf 1}\in V^{\bar{0}}$ a distinguished non-zero vector.
In what follows, we write
$$Y(A;z)=\sum_{n\in\mathbb{Z}}A_{(n)}z^{-n-1}\in \End(V)[\![z^{\pm1}]\!]$$
for $A\in V$.

\begin{dfn}\label{VSA}
A triple $(V,Y,{\bf 1})$ is a \emph{vertex superalgebra} if
\begin{enumerate}
\item
$A_{(n)}B=0$ for any $A,B\in V$ and $n\gg0$,
\item
${\bf 1}_{(n)}A=\delta_{n,-1}A$ for any $A\in V$ and $n\in\mathbb{Z}$,
\item $A_{(-1)}{\bf 1}=A$ for any $A\in V$,
\item the \emph{Borcherds identity} holds, {\it i.e.}
\begin{equation*}
\begin{array}{ll}
&\displaystyle\sum_{i\geq0}\binom{p}{i}(A_{(r+i)}B)_{(p+q-i)}\\
&\displaystyle
\hspace{4mm}
=\sum_{i\geq0}(-1)^{i}\binom{r}{i}
\left(A_{(p+r-i)}B_{(q+i)}-(-1)^{r+ab}B_{(q+r-i)}A_{(p+i)}\right)
\end{array}
\end{equation*}
for any $A\in V^{\bar{a}}$, $B\in V^{\bar{b}}$ and $p,q,r\in\mathbb{Z}$.
\end{enumerate}
Moreover, a vertex superalgebra $(V,Y,{\bf 1})$ is \emph{$\mathbb{Q}$-graded} if
\begin{enumerate}
\item
$V$ is a $\mathbb{Q}$-graded $\mathbb{C}$-vector space
, {\it i.e.} $\displaystyle V=\bigoplus_{\Delta\in\mathbb{Q}}V_{\Delta}$,
\item
$(V_{\Delta})_{(n)}(V_{\Delta'})\subset V_{\Delta+\Delta'-n-1}$
for any $\Delta, \Delta'\in\mathbb{Q}$ and $n\in\mathbb{Z}$.
\end{enumerate}
\end{dfn}

An even vector $\omega\in V_{2}$ is a \emph{conformal vector
of central charge $c\in\mathbb{C}$} if
\begin{enumerate}
\item The Virasoro algebra
${\it Vir}=\bigoplus_{n\in\mathbb{Z}}\mathbb{C}L_{n}\oplus\mathbb{C}C$ 
acts on $V$,
where the action is given by $L_{n}\mapsto\omega_{(n+1)}$ and
$C\mapsto c\id_{V}$. 
\item For any $A\in V$, we have
$Y(\omega_{(0)}A;z)=\partial_{z}Y(A;z).$
\item The operator $\omega_{(1)}$ is diagonalizable and the corresponding eigenspace decomposition
coincides with the $\mathbb{Q}$-grading of $V$.
\item For any $n\geq2$, the operator $\omega_{(n)}$ is locally nilpotent on $V$.
\end{enumerate}

In this paper, a $\mathbb{Q}$-graded vertex superalgebra together with
a conformal vector of central charge $c\in\mathbb{C}$
is called a \emph{conformal vertex superalgebra} of central charge $c$.
By abuse of notation, we write
$L_{n}$ for the linear operator $\omega_{(n+1)}\in\End(V)$.

\begin{rem}
A conformal vertex superalgebra
$(V,Y,{\bf 1},\omega)$ is called a 
\emph{vertex operator superalgebra} (VOSA)
 if ${\sf dim} V_{\Delta}<\infty$ for any $\Delta\in\mathbb{Q}$ 
and $V_{\Delta}=\{0\}$ for $\Delta\notin\frac{1}{2}\mathbb{Z}$ or $\Delta\ll0$.
A VOSA $(V,Y,{\bf 1},\omega)$ is simply called 
a \emph{vertex operator algebra} (VOA) if 
$V=V^{\bar{0}}$ and $V_{\Delta}=\{0\}$ for $\Delta\notin\mathbb{Z}$.
\end{rem}

\begin{dfn}
Let $(V^{1},Y^{1},{\bf 1}^{1},\omega^{1})$ and $(V^{2},Y^{2},{\bf 1}^{2},\omega^{2})$ 
be conformal vertex superalgebras.
A linear mapping $f\colon V^{1}\rightarrow V^{2}$ is 
a \emph{morphism of vertex superalgebra} if
$f({\bf 1}^{1})={\bf 1}^{2}$, $f(\omega^{1})=\omega^{2}$
and $f\bigl(Y^{1}(A;z)B\bigr)=Y^{2}\bigl(f(A);z\bigr)f(B)$ for any $A,B\in V^{1}$.
\end{dfn}

\begin{dfn}
Let $V$ be a vertex supralgebra and $W$ its vertex subsuperalgebra.
The \emph{commutant} (or \emph{coset}) \emph{vertex superalgebra}
${\sf Com}(W,V)$ is
the vertex subsuperalgebra of $V$ defined by
$${\sf Com}(W,V):=\bigl\{A\in V\,\big|\, A_{(n)}B=0 \text{ for any } B\in W \text{ and } n\geq0\bigr\}$$
together with $Y(-;z)|_{{\sf Com}(W,V)}$ and ${\bf 1}\in {\sf Com}(W,V)$.
\end{dfn}

\subsection{Weak modules}

Let $M=M^{\bar{0}}\oplus M^{\bar{1}}$ be
a $\mathbb{Z}_{2}$-graded $\mathbb{C}$-vector space and
$Y_{M}(-;z)\colon V\rightarrow\End(M)[\![z^{\pm1}]\!]$
an even linear mapping.
Similarly to the case of the state-field correspondence $Y(-;z)$, we write
$$Y_{M}(A;z)=\sum_{n\in\mathbb{Z}}A^{M}_{(n)}z^{-n-1}\in \End(M)[\![z^{\pm1}]\!]$$
for $A\in V$.

\begin{dfn}\label{weakmod}
A pair $(M,Y_{M})$ is a \emph{weak $V$-module} if
\begin{enumerate}
\item
$A^{M}_{(n)}m=0$ for any $A\in V$, $m\in M$ and $n\gg0$,
\item
${\bf 1}^{M}_{(n)}m=\delta_{n,-1}m$ for any $m\in M$ and $n\in\mathbb{Z}$,
\item the Borcherds identity holds, {\it i.e.}
\begin{equation*}
\begin{array}{ll}
&\displaystyle\sum_{i\geq0}\binom{p}{i}(A_{(r+i)}B)^{M}_{(p+q-i)}\\
&\displaystyle
\hspace{4mm}
=\sum_{i\geq0}(-1)^{i}\binom{r}{i}
\left(A^{M}_{(p+r-i)}B^{M}_{(q+i)}-(-1)^{r+ab}B^{M}_{(q+r-i)}A^{M}_{(p+i)}\right)
\end{array}
\end{equation*}
for any $A\in V^{\bar{a}}$, $B\in V^{\bar{b}}$ and $p,q,r\in\mathbb{Z}$.
\end{enumerate}
\end{dfn}

\begin{dfn}
Let $(M^{1},Y_{M^{1}})$ and $(M^{2},Y_{M^{2}})$ be weak $V$-modules.
A linear mapping $f\colon M^{1}\rightarrow M^{2}$ is 
a \emph{morphism of weak $V$-modules} if
$f\bigl(Y_{M^{1}}(A;z)m\bigr)=Y_{M^{2}}(A;z)f(m)$
for any $A\in V$ and $m\in M^{1}$.
\end{dfn}

We denote the category of weak $V$-modules by $V$-{\sf Mod}.
By definition, the pair $(V,Y)$ is a weak $V$-module called the \emph{adjoint module}.

\subsection{VOAs associated with affine Lie algebras}\label{Vaff}

In this subsection we work on a reductive Lie algebra $\mf{g}$ over $\mathbb{C}$
to fix some notations.
Let $B\colon \mf{g}\times\mf{g}\to\mathbb{C}$ 
be a symmetric invariant bilinear form and
 $\widehat{\mf{g}}_{B}=\mf{g}\otimes\mathbb{C}[t,t^{-1}]
\oplus\mathbb{C}K$
the corresponding affinization of $\mf{g}$,
{\it i.e.}
the commutation relations are given by
\begin{equation*}
[X_{n},Y_{m}]=[X,Y]_{n+m}+B(X,Y)\,n\delta_{n+m,0}K,\ 
[\widehat{\mf{g}},K]=\{0\}
\end{equation*}
for $X,Y\in\mf{g}$ and $n,m\in\mathbb{Z}$,
where $X_{n}:=X\otimes t^{n}\in\widehat{\mf{g}}_{B}$.
We write $\widehat{\mf{g}}_{B}^{\geq0}$ for the Lie subalgebra
$\mf{g}\otimes\mathbb{C}[t]\oplus\mathbb{C}K$ of $\widehat{\mf{g}}_{B}$.

%\cite[Theorem 2.4.1]{frenkel1992vertex}

\begin{prp}
Let $\mathbb{C}_{B}$ be a $1$-dimensional 
$\widehat{\mf{g}}_{B}^{\geq0}$-module
defined by $X_{n}.1=0$ and $K.1=1$ for $X\in\mf{g}$ and $n\geq0$.
Then there exists a unique vertex algebra structure on the
induced module
$$V_{B}(\mf{g}):=\Ind_{\widehat{\mf{g}}_{B}^{\geq0}}^{\widehat{\mf{g}}_{B}}
\mathbb{C}_{B}$$
such that $\mathbf{1}_{B}:=1\otimes1$ and
$$Y(X_{-1}\mathbf{1}_{B};z)=\sum_{n\in\mathbb{Z}}X_{n}z^{-n-1}$$
for $X\in\mf{g}$.
We call this vertex algebra
the \emph{universal affine vertex algebra associated with the affine Lie algebra
$\widehat{\mf{g}}_{B}$}.
\end{prp}

When $\mf{g}$ is simple (resp. abelian), we define $h^{\vee}_{\mf{g}}$
to be the dual Coxeter number of $\mf{g}$ (resp. $0$)
and $B_{0}$ to be the normalized symmetric invariant bilinear form
(resp.~an arbitrarily fixed non-degenerate bilinear form) on $\mf{g}$.

\begin{prp}
Suppose that $\mf{g}$ is simple or abelian.
Fix $k\in\mathbb{C}\setminus\{0,-h^{\vee}_{\mf{g}}\}$
and take a $\mathbb{C}$-basis $\{X_{a}\,|\,1\leq a\leq{\sf dim}\mf{g}\}$ of $\mf{g}$ 
and its dual basis $\{X^{a}\,|\,1\leq a\leq{\sf dim}\mf{g}\}$ of $\mf{g}$ with respect to 
$kB_{0}$.
Then
\begin{equation*}
\omega_{\mf{g},k}:=\frac{1}{2(k+h^{\vee}_{\mf{g}})}
\sum_{a=1}^{{\sf dim}\mf{g}}(X_{a})_{-1}(X^{a})_{-1}\mathbf{1}_{kB_{0}}\in V_{kB_{0}}(\mf{g})
\end{equation*}
defines a conformal vector of central charge $c_{\mf{g},k}:=\frac{k{\sf dim}{\mf{g}}}{k+h^{\vee}_{\mf{g}}}$. 
In addition, the quadruple 
$(V_{kB_{0}}(\mf{g}),Y,\mathbf{1}_{kB_{0}},\omega_{\mf{g},k})$ is a VOA.
\end{prp}

We write $L_{kB_{0}}(\mf{g})$ for the unique simple quotient VOA of $V_{kB_{0}}(\mf{g})$.
For example, see \cite[\S 2]{frenkel1992vertex} for detail.
Throughout this paper, the image of 
$A\in V_{kB_{0}}(\mf{g})$ in the quotient space $L_{kB_{0}}(\mf{g})$ again denotes $A$ by abuse of notation.

\section{Spectral flow automoriphisms}\label{SFA}

In this section we identify spectral flow twisted modules 
over $\widehat{\mf{sl}}_{2}$ and $\mf{ns}_{2}$ in \cite{FST98} 
with Li's $\Delta$-twisted modules in \cite{li1997physics}.
Moreover, we consider analogues of the spectral flows in the
Heisenberg VOAs $\mathcal{F}^{\pm}$ and lattice conformal vertex superalgebras $V^{\pm}$.

\subsection{Li's $\Delta$-automorphisms}\label{Liauto}

In this subsection,
$V$ always stands for 
$V_{\mf{sl}_{2}}$, $V_{\mf{ns}_{2}}$, $\mathcal{F}^{\pm}$ or $V^{\pm}$
(see \S \ref{Prelim} and \S \ref{Setting}).
At first, we recall the Li's theory in the super case.
According to \cite{li1997physics}, we consider
$\Delta(z)\in\End(V)\{z\}$ which satisfies the following conditions:
\begin{align*}
& \Delta(z)v\in V[z^{\pm1}]\text{ for any }v\in V, \\
& \Delta(z){\bf 1}={\bf 1}, \\
& [L_{-1},\Delta(z)]=-\partial_{z}\Delta(z), \\
& Y\bigl(\Delta(z+w)v;z\bigr)\Delta(w)
=\Delta(w)Y(v;z).
\end{align*}
Let us denote by $G(V)$ the set of elements of $\End(V)\{z\}$
satisfying the above conditions.
For simplicity, we put
$h^{V_{\mf{sl}_{2}}}:=\frac{\theta}{2}\mathbf{H}$, 
$h^{V_{\mf{ns}_{2}}}:=\theta\mathbf{J}$, 
$h^{\mathcal{F}^{\pm}}:=\lambda{\boldsymbol \alpha}^{\pm}$
and
$h^{V^{\pm}}:=r\alpha^{\pm}_{-1}$
for $\theta,r\in\mathbb{Z}$ and $\lambda\in\mathbb{C}$.

\begin{dfn}
We define
$$\Delta(h^{V};z):=z^{h^{V}_{(0)}}\exp\Bigl(\sum_{j=1}^{\infty}
\frac{h^{V}_{(j)}}{-j}(-z)^{-j}\Bigr)\in\End(V)[\![z^{\pm1}]\!].$$
\end{dfn}

\begin{lem}
For any $\theta,r\in\mathbb{Z}$ and $\lambda\in\mathbb{C}$,
we have $\Delta\left(h^{V};z\right)\in G\bigl(V\bigr)$.
\end{lem}

\begin{proof}
Note that $h^{V}_{(0)}$ has integral eigenvalues on $V$.
Though the vertex superalgebras are not regular,
the same discussion in the proof of \cite[Proposition 3.2]{li1997physics}
still works.
See also \cite[Proposition 2.1]{Ad01}.
\end{proof}

For a weight-wise admissible $\bigl(V,\mf{h}(V)\bigr)$-module $(M,Y_{M})$,
the mapping
\begin{equation*}
Y_{M}(\Delta(h^{V};z)-;z):\ V\rightarrow\End(M)[\![z^{\pm1}]\!]
\end{equation*}
defines another weight-wise admissible $(V,S_{V})$-module structure on $M$.

\begin{dfn}
The assignments
\begin{align*}
&\Obj(\mathcal{C}_{V})\rightarrow\Obj(\mathcal{C}_{V});\,
M\mapsto \Delta^{h}(M):=\Bigl(M,Y_{M}\bigl(\Delta(h;z)-;z\bigr)\Bigr),\\
&\Hom_{V}(M,N)\rightarrow\Hom_{V}\bigl(\Delta^{h}(M),\Delta^{h}(N)\bigr);\,
f\mapsto f
\end{align*}
define a functor 
$\Delta^{h}\colon \mathcal{C}_{V}\rightarrow \mathcal{C}_{V}.$
\end{dfn}

To simplify notation,
we put
$\Delta^{\theta}_{\mf{sl}_{2}}:=
\Delta^{\frac{\theta}{2}\mathbf{H}}$, 
$M^{\theta}:=\Delta^{\frac{\theta}{2}\mathbf{H}}(M)$, 
$\Delta^{\theta}_{\mf{ns}_{2}}:=\Delta^{\theta\mathbf{J}}$ and
$N^{\theta}:=\Delta^{\theta\mathbf{J}}(N)$.

\subsection{Heisenberg and Lattice case}\label{HLtw}

By calculation, it follows that
there exist a unique $\mathcal{F}^{\pm}$-module
isomorphisms such that
$\Delta^{\pm\lambda{\boldsymbol \alpha}^{\pm}}
(\mathcal{F}^{\pm})
\overset{\simeq}\longrightarrow
\mathcal{F}^{\pm}_{\lambda};\ 
\ket{0}^{\pm}
\mapsto
\ket{\lambda}^{\pm}$
and a unique $V^{\pm}$-module isomorphism such that
$\xi^{\pm}_{r}\colon (V^{\pm})^{r}:=\Delta^{r\alpha^{\pm}_{-1}}(V^{\pm})
\overset{\simeq}\longrightarrow
V^{\pm};$
$e^{n\alpha^{\pm}}\mapsto e^{(n+r)\alpha^{\pm}}.$
In this paper, we always identify
$\Delta^{\pm\lambda{\boldsymbol \alpha}^{\pm}}
(\mathcal{F}^{\pm})$ with
$\mathcal{F}^{\pm}_{\lambda}$
by the above isomorphism
 and simply write $\mathcal{F}^{\pm}_{\lambda}$.

\subsection{Affine and $\mathcal{N}=2$ case}

Let $\theta\in\mathbb{Z}$. 
First, we define
the Lie algebra automorphism ${\sf U}_{\mf{sl}_{2}}^{\theta}$
 of $\widehat{\mf{sl}}_{2}$ by
\begin{equation*}
{\sf U}_{\mf{sl}_{2}}^{\theta}(H_{n})=H_{n}+\theta K\delta_{n,0},\ 
{\sf U}_{\mf{sl}_{2}}^{\theta}(E_{n})=E_{n+\theta},\ 
{\sf U}_{\mf{sl}_{2}}^{\theta}(F_{n})=F_{n-\theta},\ 
{\sf U}_{\mf{sl}_{2}}^{\theta}(K)=K.
\end{equation*}
It is easy to verify that 
${\sf U}_{\mf{sl}_{2}}^{\theta}\circ U_{\mf{sl}_{2}}^{\theta'}
={\sf U}_{\mf{sl}_{2}}^{\theta+\theta'}$
for any $\theta,\theta'\in\mathbb{Z}$.

Next we define the Lie superalgebra automorphism
 ${\sf U}_{\mf{ns}_{2}}^{\theta}$ of $\mf{ns}_{2}$ by
\begin{equation*}
\textstyle
{\sf U}_{\mf{ns}_{2}}^{\theta}(L_{n})
=L_{n}+\theta J_{n}
+\frac{\ \theta^{2}}{6}C\delta_{n,0},\ 
{\sf U}_{\mf{ns}_{2}}^{\theta}(J_{n})
=J_{n}+\frac{\theta}{3}C\delta_{n,0},
\end{equation*}
\begin{equation*}
{\sf U}_{\mf{ns}_{2}}^{\theta}(G^{\pm}_{r})
=G^{\pm}_{r\pm\theta},\ 
{\sf U}_{\mf{ns}_{2}}^{\theta}(C)=C.
\end{equation*}
Similarly to the affine case, we have 
${\sf U}_{\mf{ns}_{2}}^{\theta}\circ U_{\mf{ns}_{2}}^{\theta'}
={\sf U}_{\mf{ns}_{2}}^{\theta+\theta'}$
for any $\theta,\theta'\in\mathbb{Z}$.

Both ${\sf U}_{\mf{sl}_{2}}^{\theta}$ and ${\sf U}_{\mf{ns}_{2}}^{\theta}$ are called 
the \emph{spectral flow automorphisms}.
These automorphisms induce the endofunctors
\begin{equation*}
\begin{array}{ccccc}
({\sf U}_{\mf{sl}_{2}}^{\theta})^{*}&:& \widehat{\mf{sl}}_{2}\text{\sf-mod} &
\rightarrow & \widehat{\mf{sl}}_{2}\text{\sf-mod}, \\
({\sf U}_{\mf{ns}_{2}}^{\theta})^{*}&:& \mf{ns}_{2}\text{\sf-mod} &
\rightarrow & \mf{ns}_{2}\text{\sf-mod}.
\end{array}
\end{equation*}
By some computations, we have the following.

\begin{lem}\label{sptw}
Let $M\in\Obj(\mathcal{C}_{V_{k}(\mf{sl}_{2})})$ and
$N\in\Obj(\mathcal{C}_{V_{c}(\mf{ns}_{2})})$.
Then the identity mappings of underlying spaces
\begin{equation*}
\id_{M}\colon({\sf U}_{\mf{sl}_{2}}^{\theta})^{*}(M)
\overset{\simeq}\longrightarrow M^{\theta},\ 
\id_{N}\colon({\sf U}_{\mf{ns}_{2}}^{\theta})^{*}(N)
\overset{\simeq}\longrightarrow N^{\theta}
\end{equation*}
give module isomorphisms.
In particular, the restrictions of the functors
$({\sf U}_{\mf{sl}_{2}}^{\theta})^{*}$ and $({\sf U}_{\mf{ns}_{2}}^{\theta})^{*}$ are 
naturally equivalent to $\Delta^{\theta}_{\mf{sl}_{2}}$
and $\Delta^{\theta}_{\mf{ns}_{2}}$, respectively.
\end{lem}

\bibliographystyle{alpha}

\bibliography{ref}

\end{document}